%% file: article2ter.tex
\documentclass{amsart}
\pdfoutput=1
\usepackage[headings]{fullpage}
\usepackage{amssymb,epic,eepic,epsfig,amsbsy,amsmath,amscd,color,amsthm}
\usepackage[all,cmtip]{xy}
\usepackage{amssymb,amsmath,amscd,color}
\usepackage{graphicx}
\usepackage{texdraw}
\usepackage{url}
\usepackage{mathrsfs}

\usepackage[utf8]{inputenc}
\usepackage{fontenc}
\usepackage{amsfonts}

\usepackage{tikz}
\tikzstyle cross=[preaction={draw=white, -, line width=6pt}]
\tikzstyle normal=[thick]
\usepackage{braids}
\usetikzlibrary{arrows,decorations.markings}

\usepackage{relsize,exscale}

\usetikzlibrary{arrows,decorations.markings}
\usetikzlibrary{cd}

\usepackage{calc}
                        \theoremstyle{plain}

\newtheorem{theorem}{Theorem}[section]

\newtheorem{lemma}[theorem]{Lemma}

\newtheorem{coro}[theorem]{Corollary}
\newtheorem{proposition}[theorem]{Proposition}
\newtheorem{prop}[theorem]{Proposition}

\newtheorem*{oquestion}{Open Question}
\newtheorem{claim}[theorem]{Claim}
\newtheorem{defn}[theorem]{Definition}

\theoremstyle{definition}

\newtheorem{rmk}[theorem]{Remark}
\newtheorem{example}[theorem]{Example}

\theoremstyle{definition}

\def\BC{\mathbb C}
\def\BN{\mathbb N}
\def\BZ{\mathbb Z}

\def\BQ{\mathbb Q}

\def\CB{\mathcal B}

\def\CG{\mathcal G}
\def\CH{\mathcal H}

\def\CM{\mathcal M}

\def\CR{\mathcal R}

\def\CU{\mathcal U}

\def\Id{\mathrm{Id}}
\def\fS{\mathfrak S}

\def\be { \begin{equation} }
\def\ee { \end{equation} }

\def\bpm { \begin{pmatrix} }
\def\epm { \end{pmatrix} }

\newcommand{\slt}{{\mathfrak{sl}(2)}}
\newcommand{\Uq}{{U_q\slt}}
\newcommand{\UqhL}{{U^{\frac{L}{2}}_q\slt}}

\newcommand{\End}{\operatorname{End}}
\newcommand{\Hom}{\operatorname{Hom}}

\newcommand{\RR}{\operatorname{R}}

\newcommand{\perm}{\operatorname{perm}}

\newcommand{\Quant}{\operatorname{Quant}}

\newcommand{\Gassner}{\operatorname{Gassner}}

\newcommand{\Vectcat}{\mathcal{V}ect}

\newcommand{\Bn}{\mathcal{B}_n}
\newcommand{\PBn}{\mathcal{PB}_n}
\newcommand{\PB}{\mathcal{PB}}
\newcommand{\Sn}{\mathfrak{S}_n}
\newcommand{\Sk}{\mathfrak{S}}

\newcommand{\Mod}{\mathop{Mod}}

\newcommand{\Forget}{\mathop{Forget}}

\newcommand{\Aut}{\mathop{Aut}}
\newcommand{\Ker}{\mathop{Ker}}
\newcommand{\Vect}{\mathop{Span}}

\newcommand{\qbin}[2]{\left[\begin{array}{c}
      #1 \\
      #2 \end{array}\right]}

\newcommand{\bapp}{\left. \begin{array}{rcl}}
\newcommand{\eapp}{\end{array} \right.}
\newcommand{\bfct}{\left\lbrace \begin{array}{rcl}}
\newcommand{\efct}{\end{array} \right.}

\newcommand{\Habs}{\operatorname{\CH}^{\text{abs}}}
\newcommand{\Hrelm}{\operatorname{\CH}^{\text{rel }-}}

\newcommand{\Laurent}{\CR}

\newcommand{\Enm}{E_{n,m}}

\title{Colored version for Lawrence representations}
\author{Jules Martel}
\date{}

\begin{document}

\maketitle

\begin{abstract}
We give an explicit isomorphism between the Gassner representation and the first weight level of a representation of quantum $\slt$. Then we construct and provide matrices for colored versions of the BKL representation and higher Lawrence's representations. 
\end{abstract}

\tableofcontents

\section{Introduction}

%We define the braid groups, and the pure braid groups.

\begin{defn}[Braid groups, Pure braid groups]\label{braidgroupdef}
Let $n\in \BN$, and let $D_n$ be the unit disk with $n$ punctures. It consists in the unit disk with $n$ points denoted $p_1 , \ldots , p_n$ (considered inside and to lie on the real line) removed. The {\em braid group} on $n$ strands is the mapping class group of the punctured disk $D_n$.
\[
\Bn = \Mod(D_n).
\]
It is a group generated by $n-1$ elements satisfying the so called {\em ``braid relations"}:
$$\Bn = \left\langle \sigma_1,\ldots,\sigma_{n-1} \Big| \begin{array}{ll} \sigma_i \sigma_j = \sigma_j \sigma_i & \text{ if } |i-j| \le 2 \\ 
\sigma_i \sigma_{i+1} \sigma_i = \sigma_{i+1} \sigma_i \sigma_{i+1} & \text{ for } i=1,\ldots, n-2 \end{array} \right\rangle$$
where the generator $\sigma_i$ corresponds to the isotopy class of the half Dehn twist swapping $p_i$ and $p_{i+1}$.
%\end{itemize}

The {\em pure braid group} on $n$ strands $\PBn$ consists in the {\em pure mapping class group} of $D_n$, namely mapping classes of homeomorphisms fixing punctures pointwise. Let $\perm : \Bn \to \Sk_n$ be the morphism assigning the permutation it involves to a given braid, then $\PBn$ is the kernel of $\perm$. 
\end{defn}

\begin{prop}[{\cite[Lemma~1.8.2]{Bir}}]
There exists a family of elements denoted $A_{i,j}$ for $1\le i < j \le n$ generating the pure braid group $\PBn$. Their expressions in terms of generators of $\Bn$ is provided in \cite[Lemma~1.8.2]{Bir}. 
\end{prop}

In \cite{Law}, R. Lawrence constructs homological representations for the braid groups using the fact that they act by homeomorphisms on punctured disks. 

\begin{defn}[Configuration space of the punctured disk]\label{configspaceofthepunctureddisk0}
Let $n,m$ be integers. The configuration space $C_{n,m}$ of $m$ unordered points in $D_n$ is defined as follows:
\[
C_{n,m} = \{ (z_1, \ldots, z_m) \in (D_n)^m \text{ s.t. } z_i \neq z_j \text{ for }  i \neq j\} / \fS_m
\] 
\end{defn}

Then the construction uses a local system:
\begin{equation}\label{Z2}
L_m: \pi_1(C_{n,m}) \to \BZ^2 
\end{equation}
that turns homology groups $H_m(C_{n,m})$ into modules over the ring of Laurent polynomials in two variables. The braid group $\Bn$ acts on $H_m(C_{n,m})$ and its action commutes with the action of Laurent polynomials. This provides a graded family of Lawrence representations for the braid groups (graded by $m \in \BN$). See \cite[Section~3.1]{Itogarside}. 

The notoriety of Lawrence's representations comes from independent work of S. Bigelow and D. Krammer (\cite{Big1}, \cite{Kra}) showing the faithfulness of braid representations at the second level of the grading ($m=2$), the one we refer to as the {\em BKL representation}. It is the first known linear and finite dimensional representation of the braid groups. 

It was proved that Lawrence's representations can be recovered as submodules of some quantum representations of quantum $\slt$, $\Uq$. In \cite{JK} for the case of BKL representation, and in \cite{Koh} for all the other level of the grading but the isomorphism does not hold on the ring of Laurent polynomials, but only generically when variables are specialized to complex values (this result is summed up in \cite[Theorem~4.5]{Itogarside}). In \cite{Jules1} Lawrence's representations are extended to relative homology modules; this allows to recover the braid group representations on the whole (tensor product of) quantum Verma modules, while over the ring ot Laurent polynomials. This result recovers Kohno's theorem (\cite{Koh}) allowing to remove genericity conditions. A first identification with quantum representation was provided by Zinno in \cite{Zi}, where he identifies the BKL representation with a quantum algebraic object, namely the quotient of the Birman-Wenzl-Murakami algebra. In \cite{MW}, they recover BKL representation as a submodule of some representations of quantum $\mathfrak{sl}(2 | 1)$. 

The goal of the present paper is to provide a concrete approach to the {\em colored Lawrence representations} by providing explicit matrices for the action of the generators of the braid groups. By colored we mean versions with $n+1$ variables instead of two. This is achieved by modifying Morphism (\ref{Z2}) by changing the target group to $\BZ^{n+1}$. These colored Lawrence representations are also defined and treated in \cite{Jules1,Kohmulti}.

The first level of Lawrence representations is known to be the Burau representation (see \cite{JK} for instance), whose colored version is called the {\em Gassner representation}, see \cite{B-N,Mo}. In Section \ref{Gassnertoquantum} we present these representations from the Magnus representations point of view (involving Fox derivatives), we provide matrices obtained from Burau matrices, and we relate them to quantum representations of braid groups arising from quantum algebra $\Uq$. 

In Section \ref{cBKL} we construct a colored version for BKL representations (level two of the grading) using the methods and basis provided in \cite{Big1}. We compute matrices in this context, using Fox calculus and Bigelow's forks--noodles pairing. 

Section \ref{lapartiereine} gives a panorama on different bases for Lawrence representations and the relations with quantum representations. We then compute matrices for colored higher Lawrence representations ($m \ge 2$) using homology techniques developed in \cite{Jules1}. We provide a computation on an integral basis, namely a basis of the entire homology module (which is not the case of the forks' basis introduced by Bigelow). This basis differs from that used in Section \ref{cBKL}, and as explained in Section \ref{lapartiereine}, it has better properties which allow to match the action of the braid group to the quantum representations while working over the ring of Laurent polynomials, see \cite[Theorem~3]{Jules1}. In section \ref{Concretecase3strands} we explain how to operate the change of basis in the case $n=3$, the general case is analogue but the notation is just cumbersome.  

The Appendix (Section \ref{pureandcolored}) explains why passing to colored representations involves restriction to the pure braid group, and suggests point of views on how to consider representations of the whole braid group or representations of the {\em colored braid groupoid}. \\

{\bf Acknowledgment} This work was achieved during the PhD of the author that was held in the {\em Institut de Mathématiques de Toulouse}, in {\em Université Paul Sabatier, Toulouse 3}. The author thanks very much his advisor Francesco Costantino for suggesting this problem, and for fruitful remarks that led to this paper.  The author is also very grateful to Emmanuel Wagner for his valuable comments and advises about this work. %The authors also thanks Emmanuel Wagner for precious comments and perspective for this work. 

\section{Gassner representation recovered by $\Uq$ representations.}\label{Gassnertoquantum}
%for  each $r$ the Pure Braid group representaition obtained from $U^H_q(sl_2)$ is conjugated to Gassner
%

In this section we present Gassner representations in Section \ref{Gassnerrep}, then some quantum representations of the braid groups in Section \ref{quantumrep}. Finally we show how to recover one from the other in Section \ref{Gassnerfromquantum}.

\subsection{Gassner representations}\label{Gassnerrep}

In this section we provide a survey about the Gassner representation, that is a ``colored version'' of the Burau representation. By colored, we mean with several variables, such that if one specializes all variables to the same one, it recovers the Burau representation. They both arises from the family of Magnus representations, defined from a representation of the braid group in the automorphism group of the free group, using {\em Fox free differential calculus}. We introduce this family following \cite{Bir}. 

\subsubsection{Magnus representations}

\begin{defn}[Fox free differential calculus]\label{Foxcalculus}
For each $j=1,\ldots ,n$ there is a map:
\[
\frac{\partial}{\partial x_j} : \BZ F_n \to \BZ F_n
\]
given by:
\[
\frac{\partial}{\partial x_j}\left( x_{\mu_1}^{\epsilon_1} \cdots x_{\mu_r}^{\epsilon_r}  \right) = \sum_{i=1}^r \epsilon_i \delta_{\mu_i,j} x_{\mu_1}^{\epsilon_1} \cdots x_{\mu_i}^{(\epsilon_i-1)/2},
\]
and 
\[
\frac{\partial}{\partial x_j}\left( \sum a_g g  \right) = \sum a_g \frac{\partial}{\partial x_j}\left( g  \right) , \text{ } g \in F_n \text{ } a_g \in \BZ,
\]
where $\epsilon_i = \pm 1$, $\delta$ is the Kronecker symbol, and $\BZ F_n$ is the the group ring of $F_n$.  
\end{defn}

Let $\Phi$ be a homomorphism acting on $F_n$ and $A_{\Phi}$ be any group of automorphisms of $F_n$ satisfying:
\[
\Phi(x) = \Phi(a(x))
\]
for each $x \in F_n$ and $a \in A_{\Phi}$. 

\begin{defn}[Magnus representation, {\cite[Theorem~3.9]{Bir}}]\label{Magnusrep}
Let $a \in A_{\Phi}$ and $\left[a \right]^{\Phi}$ be the following $n\times n$ matrix:
\[
\left[ a \right]^{\Phi} = \left[ \Phi\left( \frac{\partial (a(x_i))}{\partial x_j}  \right) \right]_{i,j} .
\]
Then the morphism:
\[
\bfct
A_{\Phi} & \to & \CM \left( n, \BZ F_n \right)\\
a & \mapsto & {\left[ a \right]}^{\Phi}
\efct
\]
is a well defined group homomorphism, called a Magnus representation. 
\end{defn}

Let $Z_n$ be the free abelian group of rank $n$ with free basis $t_1,\ldots ,t_n$ and $\mathfrak{a}$ be the following morphism:
\[
\mathfrak{a} :
\bfct
F_n & \to & Z_n \\
x_i & \mapsto & t_i
\efct .
\]

\begin{defn}[Gassner representation of the pure braid group]\label{GassnerfromMagnus}
Let $1 \le r < s \le n$ and $A_{r,s} \in \PBn$ the corresponding generator of the pure braid group on $n$ strands. Let $\left[ A_{r,s} \right]$ be the following matrix:
\[
\left[ A_{r,s} \right]^{\mathfrak{a}} = \left[ \mathfrak{a}\left( \frac{\partial (\widetilde{A_{r,s}}(x_i))}{\partial x_j}  \right) \right]_{i,j} . 
\]
Then the morphism:
\[
\bapp
\PBn & \to & \CM \left( n , \BZ \lbrack Z_n \rbrack \right) \\
A_{r,s} & \mapsto & \left[ A_{r,s} \right]^{\mathfrak{a}}
\eapp
\]
is a Magnus representation, called the Gassner representation of the pure braid group. 
\end{defn}

\begin{lemma}[{\cite[Lemma~3.11.1]{Bir}}]\label{Gassnerreducible}
The Gassner representation is reducible to an $(n-1) \times (n-1)$ representation. 
\end{lemma}
\begin{proof}[Sketch of proof]
Let $g_i = x_1 \cdots x_i \in F_n$, this provides a change of generator basis for $F_n$. The matrices:
\[
\left[ \mathfrak{a}\left( \frac{\partial (\widetilde{A_{r,s}}(g_i))}{\partial g_j}  \right) \right]_{i,j}
\]
correspond to Gassner matrices given in another basis associated to the $g_i$'s. After computation one remarks that the last rows and columns for all these matrices is $(0,\ldots,1)$ so that it can be deleted. 
\end{proof}

\begin{rmk}
Let $t=t_1= \cdots = t_n$, then the Gassner representation becomes the Burau representation. See \cite[Section~3.3]{Bir}. 
\end{rmk}

\subsubsection{Matrices}\label{matricesforGassner}

Now we give a concrete definition of the Gassner representation, with concrete matrices. They were first defined in Definition \ref{GassnerfromMagnus}, but we follow \cite{B-N} from now on to obtain matrices. 
%to build Gassner-type representations extending to braids the one of pure braids he defines in his work.\\
In \cite{B-N}, the Gassner representation is built as a ``multi-color" Burau one. Let $t$ be a formal variable and $U_{n,i}(t)$ be the standard Burau matrix associated to $\sigma_i$, the $i$'th standard generator of $\Bn$. It consists in an $n\times n$ identity matrix where one replaces the $2\times 2$ block obtained with the $i$'th and $i+1$'th rows and columns by the standard block:
\[\begin{pmatrix} 1-t && 1 \\ t && 0 \end{pmatrix}. \]
The following definition for a multivariable Burau representation is due to Morton.

\begin{defn}[\cite{B-N,Mo}]
Let $b=\prod_{\alpha=1}^{k} \sigma_{i_{\alpha}}^{s_{\alpha}}$ be a braid written as a product of standard generators. Let $\Gamma$ be the following product of matrices:
\[\Gamma(b) = \prod_{\alpha = 1}^k U_{n,i_\alpha}(t_{j_{\alpha}})^{s_{\alpha}}\]
where $j_{\alpha}$ is the index of the ``over passing'' strand at the $\# \alpha$ crossing, and $t_1,\ldots,t_n$ are set to be formal variables.
\end{defn}

\begin{prop}[\cite{B-N}]
The map:
\[
\Gamma: \bapp
\Bn & \to & \CM_n(\BZ \left[ t_i^{\pm 1} \right]_{i = 1, \ldots ,n}) \\
b & \mapsto & \Gamma(b)
\eapp
\]
is well defined. 
\end{prop}

The map $\Gamma$ is well defined but not multiplicative, i.e. not an algebra morphism. Namely, $\Gamma(ab) \neq \Gamma(a)\Gamma(b)$ when $a$ and $b$ are braids in general. 

\begin{prop}
The morphism $\Gamma$ becomes multiplicative when restricted to the pure braids, so that it yields a representation of $\PBn$.
\end{prop}

%In fact it hides an $R$-matrix, namely the standard block used to build the $U_i$'s matrices. This can be used with carefulness in the same way we deal with the $\Ubar$ $R$-matrix for quantum representation in section \ref{quantumrep}. With this remark we are able to do the same work we do to build $\Quant$ and generalise the Gassner representations to $\Bn$ entirely.

Let $\Laurent = \BZ \left[ t_i^{\pm 1} \right]_{i = 1, \ldots ,n}$. We build an induced representation of $\Bn$ over $\Laurent \lbrack \Sn \rbrack \otimes \Laurent\langle g_1 , \ldots , g_n \rangle $, where $\lbrace g_1 , \ldots , g_n \rbrace$ designates the canonical basis to write matrices in $\CM_n(\BZ \left[ t_i^{\pm 1} \right]_{i = 1, \ldots ,n})$. We define the induced Gassner representation as follows.

%Attention, on utilise comme couleur t_{\tau^{-1}(i+1)} question d'ordre de lecture et de composition. On veut que précomposer corresponde à rajouter un croisement ANTERIEUR!!!
\begin{defn}[Gassner representation of $\Bn$]
The {\em induced Gassner} (see Definition \ref{inducedpure}) representation of $\Bn$, denoted $\Gassner_n$ is defined using the following endomorphisms associated to standard generators and extended to all the braids multiplicatively.
\[
Gassner_n(\sigma_i) : \left\lbrace
\begin{array}{ccc} \Laurent \lbrack \Sn \rbrack \otimes \Laurent^n & \to & \Laurent \lbrack \Sn \rbrack \otimes \Laurent^n\\
\tau \otimes v & \mapsto & (i,i+1) \circ \tau \otimes U_{n,i}(t_{\tau^{-1}(i+1)})(v)
\end{array} \right.
\]
where $\sigma_i$ is the $i^{th}$ standard generator of $\Bn$, and $(i,i+1)$ is the permutation of $i$ and $i+1$. It's a representation over a space of dimension $n! \times n$. 
\end{defn}

This representation contains the Gassner representation of pure braids. It also contains the Burau representation as it was already the case for $\Gamma$, we state this in the following remark.

\begin{rmk}
\begin{itemize}
\item If $a$ is a pure braid, $Gassner_n(a)$ is block diagonal and $\Gamma(a)$ is the matrix restricted to $\Laurent \lbrack () \rbrack \otimes \Laurent^n$, $()$ stands for the identity permutation.
\item If all the variables are set to be equal to one variable, namely $ t_1 = \cdots = t_n = t$, then $\Gamma$ is the Burau representation.
\end{itemize}
\end{rmk}

\subsubsection{A word about faithfulness}

The Burau representation is known to be faithful for $n=2,3$, unfaithful for $n \ge 5$, and it remains an open question for $n=4$. The natural question coming from the study of Burau is if the Gassner representation is faithful, as it is richer than Burau in terms of variables. It is in fact still an open question.

This question is entirely contained in the question whether $\Gamma$ is a faithful representation of $\PBn$ or not. The explication is the following remark:

\begin{rmk}\label{faithfulnessispure}
The image of $\Laurent \lbrack () \rbrack \otimes \Laurent^n$ under the action of a braid $a$ is contained in the space $\Laurent \lbrack \perm(a) \rbrack \otimes \Laurent^n$.
This ensures that in order to get the identity matrix from $Gassner_n$, the braid $a$ must be pure.
\end{rmk}

This remark is a direct consequence of Definition \ref{inducedpure}. The faithfulness of the Gassner representation is reduced to the following open question.

\begin{oquestion}
Is $\Gamma$ faithful as a representation of $\PBn$?
\end{oquestion}

We end this presentation with a word about faithfulness of Gassner representations. We recall the Birman exact sequence {\cite[Theorem~4.6]{F-M}} in the case of the punctured disk that involves the pure braid group $\PBn$:
\[
1 \to F_{n-1} \to \PBn \to \PB_{n-1} \to 1,
\]
which is called the {\em Fadell -- Neuwirth} exact sequence. Indeed, let $D_n$ be the disk with $n$ punctures, this exact sequence is the Birman exact sequence (see {\cite[Theorem~4.6]{F-M}}) while remarking that the pure braid group is the pure mapping class group of $D_n$, and that the $\pi_1$ of $D_{n-1}$ is a free group in $n-1$ generators denoted $F_{n-1}$. Moreover this pure Birman exact sequence splits so that $\PBn$ is the semi direct product of $\PB_{n-1}$ with $F_{n-1}$. Let $\Gamma_n$ be the Gassner representation of the pure braid group $\PBn$, then one can check that the following diagram commutes:
\[
\begin{tikzcd}[column sep=small]
\PBn \arrow{r}{\hat{\Forget}}  \arrow{d}{\Gamma_n} 
  & \PB_{n-1} \arrow{d}{\Gamma_{n-1}} \\
  \Gamma_n\left( \PBn \right) \arrow{r} & \Gamma_{n-1}\left( \PB_{n-1} \right)
\end{tikzcd}
\]
where the lower horizontal arrow consists in setting $t_n$ to be 1 and deleting last row and column of the matrix. This fact allows a treatment of the faithfulness question by recursion on $n$. In some sense the Gassner representation commutes with the $\Forget$ map so that the recursion property is reduced to the faithfulness of the induced representation of $\Gamma_n$ over $F_{n-1}$ (\cite[Section~2.2]{Knu} for a presentation of these facts). It was used in a series of articles to refine the kernel of Gassner representations. The theorem giving the finest kernel the author knows is the following:
\begin{theorem}[{\cite[Theorem~3.4]{Knu}}]
The kernel of the action of $\Gamma_n$ over $F_{n-1}$ lies in $\left[ C^3F_{n-1},C^2F_{n-1} \right]$ where $C^{\bullet}F_{n-1}$ stands for the terms of the lower central series of $F_{n-1}$. 
\end{theorem}

\subsection{Quantum representations}\label{quantumrep}

In this section, we define some quantum representations for the braid groups, arising from the quantum group $\Uq$. 

\subsubsection{An integral version for $\Uq$}\label{halfLusztigversion}

%This section is independent from previous ones. We fix notations for quantum algebra objects that will be recovered by the above introduced homological modules. 

%\subsection{Half-Lusztig version for $\Uq$}\label{halfLusztigversion}

We give a first definition for $\Uq$.

\begin{defn}\label{Uqnaif}
The algebra $\Uq$ is the algebra over $\BQ(q)$ generated by elements $E,F$ and $K^{\pm 1}$, satisfying the following relations:
\begin{align*}
KEK^{-1}=q^2E & \text{ , } KFK^{-1}=q^{-2}F \\
\left[E, F \right] = \frac{K-K^{-1}}{q-q^{-1}} & \text{ and }
KK^{-1}=K^{-1}K=1 .
\end{align*}
The algebra $\Uq$ is endowed with a coalgebra structure defined by $\Delta$ and $\epsilon$ as follows:
\[
\begin{array}{rl}
\Delta(E)= 1\otimes E+ E\otimes K, & \Delta(F)= K^{-1}\otimes F+ F\otimes 1 \\
\Delta(K) = K \otimes K, & \Delta(K^{-1}) = K^{-1}\otimes K^{-1} \\
\epsilon(E) = \epsilon(F) = 0, & \epsilon(K) = \epsilon(K^{-1}) = 1
\end{array}
\]
and an antipode is defined as follows:
\[
S(E) = EK^{-1}, S(F)=-KF,S(K)=K^{-1},S(K^{-1}) = K.
\]
This provides a {\em Hopf algebra} structure, so that the category of modules over $\Uq$ is monoidal. 
\end{defn}

We define quantum numbers.

\begin{defn}\label{quantumq}
Let $i$ be a positive integer. We define the following elements of $\BZ \left[ q^{\pm 1} \right]$.
\begin{equation*}
\left[ i \right]_q := \frac{q^i-q^{-i}}{q-q^{-1}} , \text{  } \left[ k \right]_q! := \prod_{i=1}^k \left[ i \right]_q , \text{  } \qbin{k}{l}_q := \frac{\left[ k \right]_q!}{\left[ k-l \right]_q! \left[ l \right]_q!} .
\end{equation*}
\end{defn}

We are going to define an integral version for $\Uq$, so that we will obtain integral representations for braid groups. This integral version is similar to the one introduced by Lusztig in \cite{Lus}. The difference is that we consider only the divided powers of $F$ as generators, not those of $E$. This version is introduced in \cite{Hab} \cite{JK} and \cite{Jules1} (with subtle differences in the definitions for divided powers for $F$). We follow \cite{Jules1}, so that we first define the so called divided powers. Let:
\[
F^{(n)} =  \frac{(q-q^{-1})^n}{\left[ n \right]_q!} F^n .
\]
Let $\Laurent_0 = \BZ\left[ q^{\pm 1} \right]$ be the ring of integral Laurent polynomials in the variable $q$. 

\begin{defn}[Half integral algebra, \cite{JK} \cite{Jules1}]\label{Halflusztig}
Let $\UqhL$ be the $\Laurent_0$-subalgebra of $\Uq$ generated by $E$, $K^{\pm 1}$ and $F^{(n)}$ for $n\in \BN^*$. We call it a {\em half integral version} for $\Uq$, the word half to illustrate that we consider only half of divided powers as generators. 
\end{defn}

\begin{rmk}\label{relationsUqhL}
$\UqhL$ inherits a Hopf algebra structure, making its category of modules monoidal. The coproduct is given by:
\[
\Delta(K) = K \otimes K \text{ , } \Delta(E) = E \otimes K + 1 \otimes E , \text{ and } \Delta(F^{(n)}) = \sum_{j=0}^n q^{-j(n-j)}K^{j-n} F^{(j)} \otimes F^{(n-j)}. 
\]
\end{rmk}

%\begin{prop}
%The algebra $\UqhL$ admits the following set as an $\Laurent_0$-basis:
%\[
%\left\lbrace K^l E^m F^{(n)} , l \in \BZ, m,n \in \BN \right\rbrace .
%\]
%\end{prop}

\subsubsection{Verma modules and braiding}\label{VermaBraiding}

We define the {\em Verma modules} for $\UqhL$, they are infinite dimensional modules depending on a parameter. To preserve the integral structure of coefficients, we let $\Laurent_1 := \BZ \left[ q^{\pm 1} , s^{\pm 1} \right]$.
%
%\begin{defn}[Universal integral Verma modules, {\cite[\S~10.1.A]{C-P}}]\label{universalVerma}
%Let $U$ be an integral version of $\Uq$ and $s$ be a variable. The Verma module $V^{s}$ is the infinite $U$-module defined as follows:
%\[
%V^{s} = \left( U \otimes \BZ \left[ s^{\pm 1} \right]  \right) \Big/ \CI
%\]
%where $\CI$ is the left ideal generated by $E$ and $K-s1$.
%\end{defn}
\begin{defn}[Verma modules for $\UqhL$, {\cite[(18)]{JK}, \cite{Jules1}}]\label{GoodVerma}
Let $V^{s}$ be the Verma module of $\UqhL$. It is the infinite $\Laurent_1$-module, generated by vectors $\lbrace v_0, v_1 \ldots \rbrace$, and endowed with an action of $\UqhL$, generators acting as follows:
\[
K \cdot v_j = s q^{-2j} v_{j} \text{ , } E \cdot v_j = v_{j-1} \text{ and } F^{(n)} v_j = \left( \qbin{n+j}{j}_q \prod_{k=0}^{n-1} \left( sq^{-k-j} - s^{-1}q^{j+k} \right) \right) v_{j+n} .
\]
\end{defn}

\begin{rmk}[Weight vectors]\label{weightdenomination}
We will often make implicitly the change of variable $s := q^{\alpha}$ and denote $V^s$ by $V^{\alpha}$. This choice made to use a practical and usual denomination for eigenvalues for the $K$ action (which is diagonal in the given basis). Namely, we say that vector $v_j$ is {\em of weight $\alpha - 2j$}, as $K \cdot v_j = q^{\alpha - 2j} v_j$. The notation with $s$ shows an integral Laurent polynomials structure strictly speaking. 
\end{rmk}

\begin{defn}[$R$-matrix, {\cite[(21)]{JK}}]\label{goodRmatrix}
Let $s=q^{\alpha}$ , $t=q^{\alpha'}$. The operator $q^{H \otimes H /2}$ is the following:
\[
q^{H \otimes H /2}:
\bfct
V^{s} \otimes V^{t} & \to & V^{s} \otimes V^{t}  \\
v_i \otimes v_j & \mapsto & q^{(\alpha - 2i)(\alpha'-2j)} v_i \otimes v_j 
\efct .
\]
We define the following R-matrix:
\[
R : q^{H \otimes H/2} \sum_{n=0}^\infty q^{\frac{n(n-1)}{2}} E^n \otimes F^{(n)} 
\]
which will be well defined as an operator on Verma modules, see the following proposition. 
\end{defn}

\begin{prop}[{\cite[Theorem~7]{JK}}]\label{UqhLbraiding}
Let $V^s$ and $V^t$ be Verma modules of $\UqhL$ (with $s=q^{\alpha}$ and $t=q^{\alpha'}$). Let $\RR$ be the following operator:
\[
\RR:  T \circ R 
%q^{-\alpha \alpha' /2}
\]
where $T$ is the twist defined by $T(v\otimes w ) = w \otimes v$. Then $\RR$ provides a braiding for $\UqhL$ integral Verma modules. 
Namely, the morphism:
\[
Q: \bfct
\Laurent_1\left[ \Bn \right] & \to & \End_{\Laurent_1, \UqhL} \left({V^s}^{\otimes n}\right)  \\
\sigma_i & \mapsto & 1^{\otimes i-1} \otimes \RR \otimes 1^{\otimes n-i-2}
\efct
\]
is an $\Laurent_1$-algebra morphism. It provides a representation of $\Bn$ such that its action commutes with that of $\UqhL$. 
\end{prop}

We now pass to a {\em colored version}, which corresponds to taking different Verma modules in the tensor product instead of the same one. Let $V^{()} = V^{\lambda_1} \otimes \cdots \otimes V^{\lambda_n}$, and by analogy, $V^{\tau} =  V^{\lambda_{\tau(1)}} \otimes \cdots \otimes V^{\lambda_{\tau(n)}}$, for $\tau \in \Sk_n$ ($()$ designates the identity permutation). Let $\Laurent = \BZ \lbrack q^{\pm 1}, s_i^{\pm 1} \rbrack_{i  = 1 , \ldots , n}$, the morphism $Q$ extends as follows:
\[
Q(\sigma_i) \in \Hom_{\Laurent}(V^{()}, V^{(i,i+1)}).  
\]

So that if $\beta$ is pure, $Q(\beta) \in \End(V^{()})$, and:
\[
Q: \Laurent\left[ \PBn \right]  \to  \End_{\Laurent_1, \UqhL} \left(V^{()}\right)
\]
is a representation of $\PBn$. One can consider the induced (colored) representation of $\Bn$ as follows:
\[
\Quant_n : \Laurent \lbrack \Bn \rbrack \to \End_{\Laurent_1, \UqhL} \left( \Laurent \lbrack \Sk_n \rbrack \otimes V^{()} \right)
\]
noticing the following isomorphism:
\[
\bfct \Laurent \lbrack \Sk_n \rbrack \otimes V^{()}   & \to & \bigoplus_{\tau \in \Sk_n} V^{\tau} \\
\tau \otimes v^{()}_i & \mapsto & v^{\tau}_i .
\efct
\]

\subsubsection{Finite dimensional braid representations}

Although braid group representations over products of Verma modules are infinite dimensional, it turns out that they are graded by finite dimensional subrepresentations.

\begin{rmk}\label{GoodsubrepKJ}
For $r \in \BN$, the subspace $W_{n,r} = \Ker( K - \left( \prod_i s_i \right)q^{-2r})$ is such that $\Laurent \lbrack \Sk_n \rbrack \otimes W_{n,r}$ provides a sub-representation of $\Bn$.
%\item The subspace $Y_{n,r} = W_{n,r} \cap \Ker E \subset W_{n,r}$ provides a sub-representation of $\Bn$. 
We usually call $W_{n,r}$ the space of {\em subweight $r$} vectors. 
\end{rmk}

Let $\lbrace v^{\tau}_0, v^{\tau}_1, \ldots \rbrace$ be the standard basis of $V^{\tau}$. The space $W_{n,1}^{()} := () \otimes W_{n,1} \in V^{()}$ is spanned by $\left( f_1,f_2,\ldots,f_n \right)$, where the $f_i$'s are defined as follows:
\[
() \otimes f_1=v^{\lambda_1}_1 \otimes v^{\lambda_2}_0 \otimes v^{\lambda_3}_0 \otimes\cdots \otimes v^{\lambda_r}_0
\]
\[
() \otimes f_2=v^{\lambda_1}_0 \otimes v^{\lambda_2}_1 \otimes\cdots \otimes v^{\lambda_r}_0
\]
and so on, with:
\[
() \otimes f_i=v^{\lambda_1}_0 \otimes v^{\lambda_2}_0 \otimes \cdots \otimes v^{\lambda_i}_1 \cdots \otimes v^{\lambda_r}_0.
\]
These vectors are built as the tensor products of $n-1$ maximal weight vectors plus one of weight (``sub-maximal''), namely $v^{\lambda_i}_1$, inserted on the $i$-th position of the tensor product. 
%
%The vectors $f_i$'s all have the same weight (eigenvalue regarding the action of $H$): $\sum_{i=1}^{n} (\lambda_i +r-1) -2$. Then we call $W_1$ the subspace of {\em ``sub-maximal weight vectors"}. 
%
%\begin{rmk}
%From Remark \ref{ConservationKerEWeights}, the space $W_1 \otimes $ is a sub-representation of the braid group. This can be seen also directly from the expression of the $R$-matrix. 
%\end{rmk}

\subsubsection{Computation of generators' actions}

We compute the action of braid groups generators over $W:= \bigoplus_{\tau \in \Sk_n} W_{n,1}^{\tau}$. 

\begin{rmk}\label{rmatrixniveau1}
Since $E(v_0) = 0$, if $i+j \le 1$, then:
\[
R(v_i \otimes v_j) = q^{H\otimes H/2}(\Id \otimes \Id + E \otimes F^{(1)})v_i \otimes v_j .
\]
The space $W$ fulfill the conditions of this formula. 
\end{rmk}

\begin{lemma}
Let $\lbrace v^{\lambda_i}_0 , v^{\lambda_i}_1 , \ldots \rbrace$ and $\lbrace v^{\lambda_j}_0 , v^{\lambda_j}_1 , \ldots \rbrace$ be the standard basis of Verma modules $V^{s_i}$ and $V^{s_j}$ respectively ($i,j \in \lbrace 1, \ldots ,n \rbrace$). Then:
\begin{align*}
\RR(v^{\lambda_1}_0 \otimes v^{\lambda_2}_0 ) & =  v^{\lambda_2}_0 \otimes v^{\lambda_1}_0 \\
\RR(v^{\lambda_1}_1 \otimes v^{\lambda_2}_0 ) & =  s_1 v^{\lambda_2}_0 \otimes v^{\lambda_1}_1 + (s_2^2 - 1) v^{\lambda_2}_1 \otimes v^{\lambda_1}_0 \\
\RR(v^{\lambda_1}_0 \otimes v^{\lambda_2}_1 ) & =  s_2 v^{\lambda_2}_1 \otimes v^{\lambda_1}_0
\end{align*}
\begin{proof}
It is a straightforward computation from Definition \ref{goodRmatrix} and Remark \ref{rmatrixniveau1}. 
\end{proof}
\end{lemma}

\begin{coro}
The representation of $\Bn$ over $W = \Laurent \lbrack \Sk_n \rbrack \otimes W_{n,1}$ is defined by the action of generators over the basis as follows:
\[
\begin{array}{lr}
\Quant_n(\sigma_k)(\tau \otimes f_k)  =   (1-s_{\tau^{-1}(k)}^2)((k,k+1)\tau \otimes f_k) - s_{\tau^{-1}(k+1)} ((k,k+1)\tau \otimes f_{k+1})   & \\
\Quant_n(\sigma_k)(\tau \otimes f_{k+1})  =   -s_{\tau^{-1}(k)}(k,k+1)\tau \otimes f_{k}  & \\
\Quant_n(\sigma_k)(\tau \otimes f_{i}) =   (i,i+1)\tau \otimes f_{i} & \text{ if } i \neq k,k+1.
\end{array}
\]
\end{coro}

\subsection{Gassner representation from quantum ones}\label{Gassnerfromquantum}
 
%In this section, we will show that the Gassner representation is contained in the non semi-simple TQFT's representation. More precisely, the representation $\Gassner$ described in Section \ref{Gassnerrep} is algebraically the same as the representation $\Quant$ built in Section \ref{quantumrep}, for any number $n$ of strands for $\Bn$.

We recall the context of both representations, namely:
\begin{itemize}

\item from Section \ref{Gassnerrep} that $\Gassner_n$ is a representation of $\Bn$:
\[
\Gassner_n: \Laurent \lbrack \Bn \rbrack \to \End_{\Laurent} \left( \Laurent \lbrack \Sn \rbrack \otimes \Vect_{\Laurent} (g_1, \ldots , g_n ) \right)
\]
involving formal variables $t_1, \ldots, t_n$.

\item from Section \ref{quantumrep} that $\Quant_n$ is a representation of $\Bn$:
\[
\Quant_n: \Laurent \lbrack \Bn \rbrack \to \End_{\Laurent} \left( \Laurent \lbrack \Sn \rbrack \otimes \Vect_{\Laurent} (f_1, \ldots , f_n ) = W \right)
\]
involving formal variables $s_1, \ldots, s_n$.

\end{itemize}

In order to relate the representations $\Quant$ and $\Gassner$, we need first to connect variables. We use the following identification:
\[
s_i^2 = t_i %= q^{-(\lambda_i-1)} \ , \ \forall i .
\]

Then let $\Phi$ be the following morphism relating both representations:
\[
\left. \begin{array}{ccc} \Laurent \lbrack \Sn \rbrack \otimes \Vect_{\Laurent} (f_1, \ldots , f_n ) & \to & \Laurent \lbrack \Sn \rbrack \otimes \Vect_{\Laurent} (g_1, \ldots , g_n ) \\
\tau \otimes f_i & \mapsto & \frac{1-s_{\tau^{-1}(i)}^2}{\prod_{j=i}^{n}s_{\tau^{-1}(j)}}\tau \otimes g_i
\end{array} \right.
\]

%where we used the basis $\tau \otimes f_i$ and $\tau \otimes g_i$ for $\BC \lbrack \Sn \rbrack \otimes \BC^n$, $\tau$ being a permutation in $\fS_n$.

\begin{theorem}\label{Gassnerarequantum}
Gassner representations are of quantum type, namely the morphism $\Phi$ conjugates $\Quant$ to $\Gassner$, in the sense that for all $n\in \mathbb{N}$, for all $\alpha \in \Bn$ the following relation holds:
\[
\Gassner_n(\alpha)\circ \Phi = \Phi \circ \Quant_n(\alpha)
\]
\end{theorem}
\begin{proof}
Let $\sigma_k$ be a standard Artin generator of $\Bn$, $\tau \in \Sn$.

Remark that if $i$ is different from $k$ and $k+1$ then, as $\Quant$ and $\Gassner$ both act by identity over $\tau \otimes f_i$, the equality is trivial on these vectors. It remains two cases.
\begin{itemize}
\item{ \em Case 1: $i=k$.} Let's compute the two sides of the commutation equality. We begin with $\Phi \circ \Quant_n(\sigma_k)(\tau \otimes f_k)$:
%\[
%\Quant_n(\sigma_k)(\tau \otimes f_k) = (1-s_{\tau^{-1}(k)}^2)((k,k+1)\tau \otimes f_k) - s_{\tau^{-1}(k+1)} ((k,k+1)\tau \otimes f_{k+1})
%\]
%
%Then, the composition by $\Phi$ gives:
\[
\begin{array}{rcl}
\Phi \circ \Quant_n(\sigma_k)(\tau \otimes f_k) & = & \Phi \left( (1-s_{\tau^{-1}(k)}^2)((k,k+1)\tau \otimes f_k) + s_{\tau^{-1}(k+1)} ((k,k+1)\tau \otimes f_{k+1}) \right)\\
& = & A\cdot (k,k+1)\tau \otimes g_k + B \cdot (k,k+1)\tau \otimes g_{k+1}
\end{array}
\]

where:
\[
A = (1-s_{\tau^{-1}(k)}^2)\frac{1-s_{\tau^{-1}(k+1)}^2}{\prod_{j=k}^{n}s_{\tau^{-1}(j)}}
\]
\[
%\begin{array}{rcl}
B = s_{\tau^{-1}(k+1)}\frac{1-s_{\tau^{-1}(k)}^2}{s_{\tau^{-1}(k)}\prod_{j=k+2}^{n}s_{\tau^{-1}(j)}} =  s_{\tau^{-1}(k+1)}^2\frac{1-s_{\tau^{-1}(k)}^2}{\prod_{j=k}^{n}s_{\tau^{-1}(j)}}
%\end{array}
\]

Now we compute $\Gassner_n(\sigma_k)\circ \Phi(\tau \otimes f_k)$:\\
$
\begin{array}{rcl}
\Gassner_n(\sigma_k)\circ \Phi(\tau \otimes f_k) & = & \frac{1-s_{\tau^{-1}(k)}^2}{\prod_{j=k}^{n}s_{\tau^{-1}(j)}} \Gassner_n(\sigma_k)(\tau \otimes g_k)\\
& = & \frac{1-s_{\tau^{-1}(k)}^2}{\prod_{j=k}^{n}s_{\tau^{-1}(j)}} \left( (1-s_{\tau^{-1}(k+1)}^2)(k,k+1)\tau \otimes g_k \right. \\
& & \left. + s_{\tau^{-1}(k+1)}^2 (k,k+1)\tau \otimes g_{k+1}  \right) \\
& = & \Phi \circ \Quant_n(\sigma_k)(\tau \otimes f_k).
\end{array}
$

The last equality comes from the expression of $\Phi \circ \Quant_n(\sigma_k)(\tau \otimes f_k)$ obtained above, and provides the conjugation in this case.
\item{\em Case 2: $i=k+1$.} We begin with the computation of $\Phi \circ \Quant_n(\sigma_k)(\tau \otimes f_{k+1})$:
\[
\begin{array}{rcl}
\Phi \circ \Quant_n(\sigma_k)(\tau \otimes f_{k+1}) &=& \Phi\left( s_{\tau^{-1}(k)}(k,k+1)\tau \otimes f_{k}  \right)\\
&=& s_{\tau^{-1}(k)}\frac{1-s_{\tau^{-1}(k+1)}^2}{\prod_{j=k}^{n}s_{\tau^{-1}(j)}}(k,k+1)\tau \otimes g_{k}\\
&=& \frac{1-s_{\tau^{-1}(k+1)}^2}{\prod_{j=k+1}^{n}s_{\tau^{-1}(j)}}(k,k+1)\tau \otimes g_{k}
\end{array} .
\]
Now we compute $\Gassner_n(\sigma_k)\circ \Phi(\tau \otimes f_{k+1})$:
\[
\begin{array}{rcl}
\Gassner_n(\sigma_k)\circ \Phi(\tau \otimes f_{k+1}) &=& \Gassner_n(\sigma_k) \left( \frac{1-s_{\tau^{-1}(k+1)}^2}{\prod_{j=k+1}^{n}s_{\tau^{-1}(j)}} \tau \otimes f_{k} \right)\\
&=& \frac{1-s_{\tau^{-1}(k+1)}^2}{\prod_{j=k+1}^{n}s_{\tau^{-1}(j)}} \tau \otimes g_{k} \\
&=&\Phi \circ \Quant_n(\sigma_k)(\tau \otimes f_{k+1})
\end{array} .
\]
The last equality coming from the expression of $\Phi \circ \Quant_n(\sigma_k)(\tau \otimes f_{k+1})$ obtained above, and provides the desired equality.
\end{itemize}

We have proved that for any generator $\sigma_k$ of $\Bn$, its representations by $\Quant_n$ and by $\Gassner_n$ are conjugated by $\Phi$. As $\Quant_n$ and $\Gassner_n$ are representations, the theorem is proved for all braids.
\end{proof}

\begin{rmk}
The morphism $\Phi$ is invertible whenever $s_i \neq \pm 1$ for all $i \in \lbrace 1 , \ldots , n \rbrace$. \cite[Theorem~3]{Jules1} provides an isomorphism of modules preserving the integral Laurent polynomials structure of coefficients, by use of an appropriate basis for the Gassner representation.
\end{rmk}

\section{Colored BKL representations}\label{cBKL}
%Your construction

In this section, we construct BKL-like homological representations of braid groups, called {\em colored BKL representations}. We follow the construction of \cite{K-T} and \cite{Big1} for the (uncolored) BKL-representation that inspires a generalization of it. We follow ideas from \cite{Big1} to compute the matrices of this representation. This construction corresponds to the level $r=2$ of the one over modules $\Habs_r$ in \cite{Jules1}. Although the obtained representations are the same, the following construction is different: it involves Fox calculus for the computation of the homology, and uses a pairing to compute matrices. Matrices are given in Bigelow's basis of forks, we follow his work \cite{Big1}.

\subsection{Construction and Faithfulness}\label{constructionandfaithfulness}

The general concept of Lawrence's representations, see \cite{Law,Jules1}, is to make the braid group act on a homology group of a certain covering of the configuration space of several points in the punctured disk. We recall Definition \ref{configspaceofthepunctureddisk0} of the configuration spaces of points.

\begin{defn}[Configuration space of the punctured disk]\label{configspaceofthepunctureddisk}
Let $n,m$ be integers. The configuration space $C_{n,m}$ of $m$ unordered points in $D_n$ is defined as follows:
\[
C_{n,m} = \{ (z_1, \ldots, z_m) \in (D_n)^m \text{ s.t. } z_i \neq z_j \text{ for }  i \neq j \} / \fS_m
\]
where $\fS_m$ acts by permutation on the order of coordinates. 

Let $C : = C_{n,2}$ for $n$ fixed (we omit $n$ in this notation whenever no confusion arises). 
\end{defn}

We denote $\{x,y \}$ an element of $C$ ($\{x,y \} =\{y,x \}$), and $c = \{d_1,d_2\}$ a base point of $C$ with the $d_i$'s lying in the boundary of $D_n$.

\begin{prop}[{\cite[Proposition 1.3]{P-P}}]
The first homology group of $C$, namely $H_1(C, \BZ)$ is isomorphic to $\BZ ^n \oplus \BZ$.
\end{prop}

\begin{defn}\label{localsystemcoloredBKL}
We consider the Hurewicz morphism:
\[
Hurewicz: \pi_1(C) \to H_1(C) = \BZ ^n \oplus \BZ = \langle q_1 \rangle \oplus \cdots \oplus \langle q_n \rangle \oplus \langle t \rangle,
\]
and we denote by $\widetilde{C}$ the covering corresponding to the kernel of this map, namely the maximal abelian cover. 
\end{defn}

The homology group $H_2(\widetilde{C})$ is turned into a $\BZ \left[ q_1^ {\pm1} , \ldots , q_n^ {\pm1} , t^ {\pm1} \right]$-module. We are  going to show that, in Lawrence's construction spirit, this homological module is acted upon by $\PBn$ by  $\BZ \left[ q_1^ {\pm1} , \ldots , q_n^ {\pm1} , t^ {\pm1} \right]$-module automorphisms. This action is the so called {\em colored BKL representation}; it is shown in the general case of Lawrence representations (all level of gradings) in \cite[Lemma~6.34]{Jules1}. Here we follow \cite{Big1}.

\begin{rmk}
In Definition \ref{localsystemcoloredBKL}, we keep the $n+1$ generators of the abelianized group, while the {\em uncolored} version from \cite{Big1} consists in post-composing this Hurewicz map with an augmentation morphism, sending the $n$ generators $q_i$'s to a single one. 
\end{rmk}

A path $\xi: I \to C$ is a pair of paths $\xi = \{ \xi_1 , \xi_2 \}$ where $\xi_1,\xi_2 : I \to D_n$. As we are looking to unordered pairs of points, there are two possibilities for a path $\xi$ to be a loop:
\[
\xi_1(0) = \xi_1(1) \text{ and } \xi_2(0) = \xi_2(1)
\]
so that both $\xi_i$'s are loops, or:
\[
\xi_1(0) = \xi_2(1) \text{ and } \xi_2(0) = \xi_1(1),
\]
where $\xi_1$ and $\xi_2$ permutes their endpoints (they are not loops) but the product $\xi_1 \xi_2$ is a loop. 

We define invariants $w_i$ of homotopy classes of loops in $C$ for all $i \in \{1,\ldots,n\}$ and for the two cases of a loop $\xi = \{ \xi_1 , \xi_2 \}$ of $C$:
\begin{itemize}
\item If $\xi_1$ and $\xi_2$ both are loops, then we define $w_i(\xi)= w_i(\xi_1) + w_i(\xi_2)$ where $w_i(\xi_k)$ is the winding number of $\xi_k$ ($k=1,2$) around the puncture $p_i$.
\item For the case where $\xi_1$ and $\xi_2$ permute base points we define $w_i(\xi) = w_i(\xi_1 \xi_2 ) $ to be the winding number around the puncture $p_i$ of the loop $\xi_1 \xi_2$.
\end{itemize}

%Let also $u$ be the same invariant as in Section \ref{BKL}, namely the index (speaking of a loop of $S^1$) of the square of the following application:
%\[
%\begin{array}{rcl}
%I & \to & S^1 \\
%s & \mapsto & \frac{\xi_1(s) - \xi_2(s)}{|\xi_1(s) - \xi_2(s)|}
%\end{array}
%\]
%%sends $s=0,1$ to the same numbers or opposite ones. So the square of this function provides a loop of $S^1$, $u(\xi)$ is the index of it. Note that $u(\xi)$ is even if the $\xi_i$'s are loops, odd otherwise. These classic invariants are additive with respect to product of loops and preserved under homotopy.

We define another invariant $u$, remarking that the map:
\[
\left\lbrace \begin{array}{rcl}
I & \to & S^1 \\
s & \mapsto & \frac{\xi_1(s) - \xi_2(s)}{|\xi_1(s) - \xi_2(s)|}
\end{array} \right.
\]
sends $s=0,1$ to the same points or to opposite ones. Hence, the square of this function provides a loop of $S^1$, $u(\xi)$ is the index of it. Note that $u(\xi)$ is even if the $\xi_i$'s are loops, odd otherwise. These classic invariants are additive with respect to product of loops and preserved under homotopy.

These invariants can equivalently be defined as follows:
\[
w_i(\xi) = \frac{1}{2\pi i}  \left( \int_{\xi_1} \frac{dz}{z-p_i} + \int_{\xi_2} \frac{dz}{z-p_i} \right)
\]
and:
\[
u(\xi) = \frac{1}{\pi i} \int_{\xi_2-\xi_1} \frac{dz}{z} .
\]

The map:
\[
\phi: \xi \to q_1^{w_1(\xi)} \cdots q_n^{w_n(\xi)} t^{u(\xi)}
\]
is a surjective group homomorphism from $\pi_1(C)$ to the free abelian group with $(n+1)$ generators $q_1,\ldots,q_n,t$. It corresponds to the Hurewicz map, see the introduction of \cite{Koh}.

Then $\widetilde{C} \to C$ is the covering map corresponding to the kernel of $\phi$, and $\CH = H^{lf}_2(\widetilde{C}, \BZ)$ is a module over $\Laurent=\BZ \left[ q_1 ^{\pm 1}, \ldots , q_n ^{\pm 1},t ^{\pm 1} \right]$, once we choose a lift $\tilde{c}$ of the base point $c$ to $\widetilde{C}$. The letters $lf$ indicate the locally finite version of the singular homology, see \cite[Appendix]{Jules1} for precisions.

We recall that if $f$ is a self-homeomorphism of $D_n$ (which is the identity on the boundary), it induces a homeomorphism $\hat{f} : C \to C$ by:
\begin{equation}\label{homeocoordinatebycoordinate}
\hat{f} ( \{ x, y \}) = \{ f(x),f(y) \}
\end{equation}
Note that $\hat{f} (c) = c$ as $d_1$ and $d_2$ are picked in the boundary of $D_n$. We define the induced automorphism  $f_{\#}$ of $\pi_1(C,c)$. The following holds.

\begin{lemma}\label{phipreserved2}
Let $f$ be such a self-diffeomorphism but fixing punctures pointwise. Then $\phi \circ f_{\#} = \phi$.
\end{lemma}
\begin{proof}
We need to verify that the invariants $w_i$ and $u$ are preserved by $f_{\#}$. %For $u$ see lemma \ref{phipreserved}, as the definition is the same.  

For $u$, $u \circ f_{\#} = u$ holds because this invariant does not ``see" the punctures, i.e. $u$ factors through the embedding $D_n\to D^2$, where $D^2$ designates the (unpunctured) unit disk. Forgetting the punctures, all homeomorphisms are isotopic to the identity (Alexander trick \cite{F-M}), so that $u \circ f_{\#} = u$.

For $w_i$, it comes from the fact that the equality $w_i \circ f_{\#} = w_i$ holds for small loops encircling the punctures, and then for arbitrary loops since it only depends on the homology class in the first homology group of $D_n$ which is generated by these small loops.

Here, the difference with construction in \cite{Big1} is that we need $f$ to fix the punctures. Otherwise, a small circle encircling a puncture transported by $f$ could count $+1$ for a different winding number before and after the application of $f$.

\end{proof}

This lemma implies that, in the case where $f$ does not permute punctures, $\hat{f}$ uniquely lifts to a map $\tilde{f}: \widetilde{C} \to \widetilde{C}$ fixing any lift of $c$, and that $\tilde{f}$ commutes with covering deck-transformations. Therefore it induces an $\Laurent$-linear automorphism $f_{*}$ of $\CH$, that is an invariant of the isotopy class of $f$. Consequently, it defines a representation of the pure braid group.  

\begin{defn}[Colored BKL representation]\label{coloredBKLrep}
The {\em colored Bigelow--Krammer--Lawrence} representation of the pure braid group is:
\[
\Laurent\lbrack \PBn \rbrack \to \Aut _{R}(\CH) \text{ , } \left[ f \right] \mapsto f_{*}
\]
where $\PBn$ refers to the pure mapping class group of the punctured disk, which corresponds exactly to (isotopy classes of) homeomorphisms fixing punctures pointwise.
\end{defn}

What follows immediately, is that by specializing every variables $q_i$ to the same variable $q$ we obtain the BKL representation of $\PBn$ as a subgroup of $\Bn$, so that the following holds.

\begin{proposition}[\cite{Kra} \cite{Big1}]
The colored BKL representation of $\PBn$ is faithful.
\end{proposition}

To get a representation of the whole braid group, one has to consider $\CH \otimes \Laurent \lbrack \Sk_n \rbrack$, see Section \ref{pureandcolored}. 

\begin{defn}[Colored BKL representation]\label{definitioncoloredBKL}
The {\em colored Bigelow--Krammer--Lawrence} representation of the braid group is:
\[
\Laurent\lbrack \Bn \rbrack \to \Aut _{R}(\CH \otimes \Laurent \lbrack \Sk_n \rbrack).
\]
\end{defn}

\subsection{Pairing between forks and noodles}\label{pairingbetweenforkandnoodles}

%We will stick to the definition of {\em forks} and {\em noodles} used in Bigelow's and Krammer's works, and will use it to compute a pairing between them taking values in $R$, that will be usefull to compute matrices. We define surfaces of $\widetilde{C}$ in the same fashion as we did in the BKL section. Say $F$ is a fork, and $N$ a noodle. We stick with $\Sigma (F)$ and $\Sigma (N)$ the surfaces of $C$ defined in section \ref{BKL}. And we lift them to $\widetilde{C}$ in the same way we did in \ref{BKL} (the covering changed but the way to lift stay the same) to obtain associated surfaces $\widetilde{\Sigma} (F)$ and $\widetilde{\Sigma} (N)$ of $\widetilde{C}$.

We recall definitions of {\em forks} and {\em noodles}, together with a pairing between these objects. Everything is adapted from \cite{Big1}.

\begin{defn}[Fork, $m=2$]\label{forksm=2}
A {\em fork} is  an embedded tree $F\in D_n$ with four vertices $d_1, p_i, p_j$, and $z$ such that $F \cap \partial D_n = \{ d_1 \}$, $F$ intersects the punctures only in $p_i, p_j$, and all three edges have $z$ as a vertex. 
\begin{itemize}
\item The edge containing $d_1$ is called the handle of $F$ and denoted $H(F)$.
\item The union of other two edges is called the {\em tine} of $F$ and denoted $T(F)$.
\item The tine is oriented in such a way that it has the handle lying on its right.
\end{itemize}
\end{defn}

For any fork $F$ we construct an associated surface $\widetilde{\Sigma}$ in $\widetilde{C}$ as follows. First let $F'$ be the parallel fork of $F$ with a parallel tine with same endpoints and parallel handle based on $d_2$.
We define the following surface of $C$:
\[
\Sigma(F) = \left\{ \{x,y\} \text{ s.t. } x  \in T(F)\backslash \{ p_1, \ldots, p_n\} \text{ } ,  y \in T(F') \backslash \{ p_1, \ldots, p_n\}  \right\} .
\]
In order to get a surface of $\widetilde{C}$ we need to choose a lift of $\Sigma(F)$. We use the handle to do so. Let $\tilde{\beta}$ be the lift beginning at $\tilde{c}$ of $\{ \beta_1 ,\beta_2 \}$ where $\beta_1, \beta_2$ are respectively the handle of $F$ and $F'$ starting on $d_1$ and $d_2$. Let $\widetilde{\Sigma}(F)$ be the lift of $\Sigma(F)$ which contains $\tilde{\beta}(1)$. %This will be call the {\em handle process} in the general set-up of Section \ref{homologicalfamily}.

\begin{defn}[Noodle]
A {\em Noodle} is an arc embedded in $D_n$ going from $d_1$ to $d_2$.
\end{defn}

We construct a surface associated to $N$ as follows:
\[
\Sigma(N) = \left \lbrace \{x,y \} \in C \text{ s.t. } x,y \in N \right \rbrace ,
\]
and then we choose $\widetilde{\Sigma} (N) $ to be the lift of $\Sigma(N)$ which contains $\tilde{c}$.

Let $F$ be a fork and $N$ a noodle, and let $\widetilde{\Sigma} (F)$ and $\widetilde{\Sigma} (N)$ the associated surfaces of $\widetilde{C}$. Suppose that $T(F)$ and $N$ intersect transversely in some points $z_1 , \ldots, z_l$, and $T(F')$ and $N$ intersect transversely in $z'_1, \ldots, z'_l$ such that $z_i$ and $z'_i$ are joint by a short piece of $N$ not containing any other intersection point. Surfaces $\widetilde{\Sigma}(F)$ and $\widetilde{\Sigma} (N)$ do not intersect necessarily because of the choice of the lift, but there exists a unique monomial $m_{i,j} = \prod_{k\in \{ 1, \ldots, l\}} q_k^{w_k(\xi_{i,j})} t^{u_{i,j}}$ such that $m_{i,j} \widetilde{\Sigma} (N) $ intersects $\widetilde{\Sigma}(F)$ at a point lying over $\{ z_i , z'_j\}$. Let $\epsilon _{i,j}$ be the sign of the intersection. We define the pairing as follows:
\begin{equation}\label{noodleforkpairingCBKL}
\langle N,F \rangle = \sum_{i=1}^{l} \sum_{j=1}^{l} \epsilon_{i,j} m_{i,j} .
\end{equation}

To compute explicitly $m_{i,j}$ we define a path of $\tilde{C}$ using composition of the following arcs:

\begin{itemize}
\item $\alpha _1$ from $d_1$ to the handle of $F$, $\alpha_2$ from $d_2$ to the handle of $F'$,
\item $\beta_1$ from $z$ to $z_i$ along $T(F)$, $\beta_2$ from $z'$ to $z'_j$ along $T(F')$,
\item $\gamma_1$ from $z_i$ to one of the $d_i$'s in such a way that it doesn't cross $z'_j$,
\item $\gamma_2$ from $z'_j$ to one of the $d_i$'s in such a way that it doesn't cross $z_i$.
\end{itemize}

Then we define the loop $\delta_{i,j}$ of $C$:
\[
\delta_{i,j} = \{\alpha_1,\alpha_2 \} \{\beta_1,\beta_2 \} \{\gamma_1,\gamma_2 \}
\]

Let $\widetilde{\delta}_{i,j}$ be the lift of $\delta_{i,j}$ to $\widetilde{C}$ beginning at $\tilde{c}$. This path goes first from $\widetilde{c}$ to $\widetilde{\Sigma}(F)$ then to the lift of $\{z_i, z'_j \}$ lying over $\widetilde{\Sigma} (F) \cap m_{i,j} \widetilde{\Sigma} (N)$, so that it ends in $m_{i,j} \tilde{c}$. It is a path from $\tilde{c}$ to $m_{i,j} \tilde{c}$, so that we have:
\[
m_{i,j} = \phi (\delta_{i,j}).
\]

From \cite[Claim 3.3]{Big1}, we have:

\[
\epsilon_{i,j} = -m_{i,i} m_{j,j} m_{i,j} (q=1,t=1)  .
\]
as the intersection sign is computable in $C$ (does not depend on which covering one lifts the surfaces to), it is the same as for BKL representations, see \cite[Equation~(1)]{Big1}.

Bigelow's proof of the faithfulness involves the following two lemmas.

\begin{lemma}[Basic Lemma, \cite{Big1}]\label{basiclemma}
The noodle-fork pairing is well defined. Furthermore, if $\left[ \sigma \right]$ lies in the kernel of the Bigelow Krammer-Lawrence representation, then
\[
\langle N,F \rangle = \langle N, \sigma(F) \rangle
\]
for any fork $F$ and noodle $N$.
\end{lemma}

This will allow us to use this pairing for computing matrices, see next section. The following key lemma is stated for informal reasons.

\begin{lemma}[Key Lemma, \cite{Big1}]\label{keylemma}
Let $N$ be a noodle and let $F$ be a fork. Then $\langle N,F \rangle = 0$ if and only if $N$ and $T(F)$ do not intersect (up to isotopy).
\end{lemma}

\subsection{Matrices for colored BKL representations}

Inspired by Part 4 of \cite{Big1} we give explicit matrices for colored BKL representations.

\begin{proposition} \label{actioncBKL}
%Consider $\Lambda$ as a subring of $\BR$ by assigning algebraically independent values to $q_k$'s and $t$, so that $H_2(\widetilde{C})$ embeds into the vector space $\BR \otimes H_2(\widetilde{C})$. 

$\CH$ is a free $\Laurent$-module. It has a basis:
\[
\{ v_{j,k} : 1\le j \le k \le n \}
\]
The group $\Bn$ acts on $\CH \otimes \Laurent \left[ \Sk_n \right]$ by the induced action from $\PBn$. We give the action of the standard generators $\sigma_i$ on $\CH \otimes 1$, let $\tau = (i,i+1)$. 
\[
\sigma_i(v_{j,k}\otimes 1) = \left \lbrace \begin{array}{ll}
v_{j,k} \otimes \tau, & i \not\in \{j-1,j,k,k-1,k\} \\
q_{j-1} v_{i,k} \otimes \tau + (q_{j-1}^2-q_{j-1})v_{i,j} \otimes \tau + (1-q_{j-1})v_{j,k} \otimes \tau , & i=j-1 \\
v_{j+1,k} \otimes \tau , & i=j\neq k-1 \\
q_{k-1} v_{j,i} \otimes \tau +(1-q_{k-1})v_{j,k}  \otimes \tau + (q_{k-1}^2-q_{k-1})tv_{i,k}  \otimes \tau, & i = k-1 \neq j \\
v_{j,k+1} \otimes \tau, & i=k \\
tq_{j}^2 v_{j,k}  \otimes \tau, & i=j=k-1
\end{array} \right. .
\]
One remarks that the only variable appearing above is $q_i$. We wanted to emphasize the links with indexes of vectors, which could be useful to follow permutations of punctures and variables. 
\end{proposition}

\begin{proof}

First we compute the homology using the Cayley complex defined in \cite[Section~4]{Big1}. For $j= 1, \ldots, n$, let $\zeta_j$ be the loop based at $d_1$ and running once counterclockwise around $p_j$. Let $x_j$ be the loop $\{ \zeta_j, d_2\} $ of $C$. Let $\tau_1$ be an arc from $d_1$ to $d_2$ and $\tau_2$ from $d_2$ to $d_1$ such that $\tau_1 \tau_2$ is a simple closed curve oriented counterclockwise and enclosing no puncture points, and let $y$ be the loop $\{ \tau_1, \tau_2 \}$ of $C$. We define the set $\CG$:
\[
\CG = \{ x_1, \ldots, x_n , y \}.
\]
Now we define some relations, for $j \in \{ 1, \ldots, n\}$:
\[
r_{j,j} = \left[x_j,y x_j y \right],
\]
and for $1\le j < k \le n$:
\[
r_{j,k} = \left[ x_j, y x_k y^{-1} \right]
\]
where the bracket refers to the commutator, and we define the set $\CR = \{ r_{j,k} \text{ for } 1\le j \le k \le n\}$.

\begin{proposition}[{\cite[Section~4]{Big1}}]\label{pi1C}
Let $K$ be the Cayley Complex of $\langle \CG | \CR \rangle $. Then $C$ is homotopically equivalent to $K$. It follows that a presentation of $\pi_1(C)$ is given by: $\langle \CG | \CR \rangle$.
\end{proposition}

%We recall the Cayley complex $K$ of the presentation $\langle \CG | \CR \rangle$, which is homotopy equivalent to $C$, from \cite{Big1}.
%
%For $j= 1, \ldots, n$, we let $x_j$ be the loop $\{ \zeta_j, d_2\} $ of $C$ and $y$ be the loop $\{ \tau_1 \tau_2 \}$ of $C$. The set $\CG$ is defined as follows:
%\[
%\CG = \{ x_1, \ldots, x_n , y \}.
%\]
%The set of relations is $\CR = \{ r_{j,k} \text{ for } 1\le j \le k \le n\}$, with for $j \in \{ 1, \ldots, n\}$:
%\[
%r_{j,j} = \left[x_j,y x_j y \right].
%\]
%and for $1\le j < k \le n$:
%\[
%r_{j,k} = \left[ x_j, y x_k y^{-1} \right].
%\]
%
%\begin{proposition}[\cite{Big1}]
%Let $K$ be the Cayley Complex of $\langle \CG | \CR \rangle $. Then $C$ is homotopically equivalent to $K$. It follows that a presentation of $\pi_1(C)$ is given by: $\langle \CG | \CR \rangle$.
%\end{proposition}

Now we can compute $\CH$ using the Fox derivatives (see Definition \ref{Foxcalculus}). We let $C _1$ and $C _2$ be the free $\Laurent$-modules with basis $\{ e_g : g \in \CG \}$ and $\{ f_r : r \in \CR \}$ respectively. For any word in $\CG$, we define $\left[ w \right] \in C_1$ according to these rules:

\[
\begin{array}{rcl}
\left[ 1  \right] & = & 0 \\
\left[ gw  \right] & = & e_g + \phi(g) \left[w \right]  \\
\left[ g^{-1}w  \right] & = & \phi(g)^{-1} (\left[ w \right] - e_g)
\end{array}
\]
for $g \in \CG$. Then $H_2 (\widetilde{C})$ is the kernel of the map $\partial: C_2 \to C_1$ defined by $\partial f_r = \left[ r \right]$.

The computation gives:

\[
\partial f_r = \left \lbrace \begin{array}{ll}
(q_jt+1)((1-t)[x_j]+(q_j-1)[y]) & \text{ if } r=r_{j,j} \\
(1-q_k)[x_j]+(1-q_k)(q_j-1)[y]+t(q_j-1)[x_k] & \text{ if } r=r_{j,k}
\end{array} \right. .
\]

If we restrict the morphism to the space $\Vect (f_{j,j},f_{j,k},f_{k,k})$, we get the matrix:
\[
\begin{pmatrix}
(1-t)(q_jt+1) & (1-q_k) & 0 \\
(q_j-1)(q_jt+1) & (1-q_k)(q_j-1) & (q_k-1)(q_kt+1) \\
0 & t(q_j-1) & (1-t)(q_kt+1)
\end{pmatrix}
\]
which corresponds to the only non-vanishing block of the application $\partial$. Each block has a rank one kernel generated by the vector:
\[
v_{j,k} = -(1-q_k)(q_kt+1)f_{j,j} + (1-t)(q_kt+1)(q_jt+1)f_{j,k} - t(q_j-1)(q_jt+1)f_{k,k}
\]
so that we get a basis of $H_2(\widetilde{C})$, namely $\{ v_{j,k} : 1 \le j < k \le n \}$.

Now a nice way to compute the matrices for the action of $\sigma_i$, is to find forks $F_{j,k}$ which correspond to the vector $v_{j,k}$, and to use the pairing with some noodles to get the expression of vectors in the fork basis, knowing the basic lemma (Lemma \ref{basiclemma}). In what follows we still abusively use $F$ to designate both the fork and the associated homology class of $\widehat{\Sigma} (F)$.

Let's fix $d_1$ and $d_2$ lying in the lower half plane of the boundary of $D_n$. 

\begin{defn}[Standard fork]\label{standardforkcBKL}
For each $1 \le j < k \le n$, let $F_{j,k}$ be the fork that lies entirely in the lower half of $D_n$ such that the endpoints of $T(F_{j,k})$ are the punctures $p_j$ and $p_k$, we usually call it {\em a standard fork}.
\end{defn}

\begin{rmk}\label{forkissue}
In \cite[Section~7.1]{Jules1} one can find a generalization of forks to higher configuration spaces of points ($m\ge 2$), that can also be found in \cite{Itogarside}. We want to warn the reader that Bigelow's standard forks defined above are {\em not} multiforks according to the definitions from \cite{Jules1,Itogarside}. Indeed in the multiforks the segments only connect consecutive punctures (but maybe different ones) while Bigelow's standard forks can join any two different punctures (but the second configuration point is then taken in the parallel fork, having same punctures as ends). 
\end{rmk}

\begin{rmk}\label{thisremark}
There exists $\lambda \in \Laurent$ such that for all $j,k \in \lbrace 1 , \ldots , n \rbrace$:
\[ F_{j,k} = \lambda v_{j,k} \]
(in terms of homology classes). The proof of this fact is exactly the same as the one for the uncolored version, see \cite{Big1} proof of Theorem 4.1. This is achieved noticing that it is sufficient to consider the homology module restricted to the disk containing $F_{j,k}$, its endpoints, and no other puncture. 
\end{rmk}

By Remark \ref{thisremark}, we compute the braid action over standard forks. There are cases where $\sigma_i(F_{j,k})$ is directly a standard fork, namely:
%$
\begin{itemize}
\item $i \not\in \{ j-1,j,k-1,k \}$
\item $i=j \neq k-1$
\item $i=k$
\end{itemize}
%$
In the case $i=j=k-1$, the fork $\sigma_i(F_{j,k})$ has the tine edges as $F_{j,k}$ with opposite orientations:
\[
\sigma_i(F_{j,k}) = \left(\vcenter{\hbox{ \begin{tikzpicture}[decoration={
    markings,
    mark=at position 0.5 with {\arrow{>}}}
    ]
\coordinate (w0) at (-1,0) {};
\coordinate (w1) at (1,0) {};
\coordinate (x0) at (-1,-1) {};
\coordinate (x1) at (1,-1) {};

\draw[postaction={decorate}] (w1) -- node[above,pos=0.5] (k0) {} (w0);
\draw[postaction={decorate}] (w1) to[bend left=40] node[pos=0.2] (k1) {} (w0);

\draw[red] (k0) arc (180:0:1) node[above] (k0p) {};
\draw[red] (k0p) to (k0p|-x0);
\draw[red] (k1) arc (180:0:0.6) node[above] (k1p) {};
\draw[red] (k1p) to (k1p|-x1);

%\node[right,red] at (a) {$\alpha$};

\node[gray] at (w0)[left=5pt] {$p_k$};
\node[gray] at (w1)[right=5pt] {$p_j$};
\foreach \n in {w0,w1}
  \node at (\n)[gray,circle,fill,inner sep=3pt]{};

\end{tikzpicture} }}\right) 
\]
where in red is represented the handle, and in black the tine. The handle joins the boundary in $d_1$ outside the parenthesis. 
It follows that it represents the same surface in $C$ as $F_{j,k}$ with the same orientation (both interval are reversed). Then the classes in $\widetilde{C}$ differ by a covering transformation. We obtain that $\sigma_i(v_{j,k}) = t q_j^2 v_{j,k} $. A similar computation is made in \cite{Jules1}, in Example 4.2, using the {\em handle rule} introduced in Remark 4.1 and recalled in the present paper in Remark \ref{handleruless}, that deals with a change of handle for forks. The remaining cases are $i=j-1$ and $i=k-1 \neq j$. The following claim restricts the linear combination, and is proved exactly the same way as Claim 4.2 of \cite{Big1}:
\begin{claim}[{\cite[Claim~4.2]{Big1}}]\label{forkdecomposition}
$\sigma_i(v_{j,k})$ is a linear combination of $v_{j',k'}$ with $j',k'\in \{i,i+1,j,k\}$
\end{claim}
In the case $i=j-1$ for instance, this claim implies that there exists $A,B,C \in \Laurent$ s.t. :
\[
\sigma_i(F_{j,k}) = AF^{\tau}_{i,j}+ BF^{\tau}_{j,k}+ CF^{\tau}_{i,k}
\]
for $\tau = (i,i+1) \in \Sk_n$, where we denote $F^{\tau}_{j,k} : = F_{\tau(j), \tau(k)}$. One remarks that punctures are permuted after the application of $\sigma_i$, and for the colored version issue following the permutation of punctures is important. 
%$\sigma_i(F_{j,k})$ represents the same element in $H_2(\widetilde{C})$ as a linear combination of the standard forks $F_{i,j},F_{j,k}$ and $F_{i,k}$. The nice tool 
To get $A,B,C$ we pair with noodles. As it only depends on homological class of the surface associated to fork, by pairing some appropriate noodles with the studied forks on one hand and with the standard fork involved in its decomposition on the other, we are able to compute the coefficients of the linear combination. In following Example \ref{examplepairing}, we perform this computation in one of the two remaining cases. One can adapt the computation to the last case. 
\end{proof}

\begin{example}\label{examplepairing}
Let $F$ be the fork corresponding to the image of $F_{2,4}$ after applying the homeomorphism corresponding to the generator $\sigma_1$ of $\Bn$. Considering the Claim \ref{forkdecomposition} we can restrict ourselves to $B_4$ and the study of $D_4$ with only four punctures. This example is enough to deduce the general expression of the action of $\sigma_i$ on the vector $v_{j,k}$ in the case $i=j-1$, which is one of the two remaining cases not entirely treated in the proof of Proposition \ref{actioncBKL}.

First we use Claim \ref{forkdecomposition} to deduce that the class in $H_2(\widetilde{C})$ associated with $F$ has a linear decomposition in terms of standard forks $F^{\tau}_{1,2},F^{\tau}_{1,4}$ and $F^{\tau}_{2,4}$, for $\tau = (1,2) \in \Sk_4$. We use the following notations:
\[
F = AF^{\tau}_{1,2} + BF^{\tau}_{1,4} + CF^{\tau}_{2,4}
\]
where $A,B,C \in \Laurent$ are the coefficients we are looking for. We compute $A$, $B$, $C$ using Pairing \ref{noodleforkpairingCBKL}.

\begin{rmk}\label{remarkmovie}
In order to compute invariants of loops $\delta_{i,j}$'s (see Subsection \ref{pairingbetweenforkandnoodles}), it is useful to draw both paths ($\xi_1,\xi_2$) to observe the value of the invariants $w_i$ while for the invariant $u$ the parametrization is crucial, so that we need to think about the movie of the loop. We draw some in Figure \ref{deltas}.
\end{rmk}

Let $N_i$ be the noodle starting at $d_1$ and passing once clockwise around the puncture $p_i$ before coming back to $d_2$. We get the easy following computation of the pairing with standard forks:

\begin{rmk}
%\begin{itemize}
%\item
\[
\langle N_i , F_{j,k} \rangle = \left \lbrace \begin{array}{ll}
-q_i & \text{if } i=j \\
q_i^{-1}t^{-1} & \text{if } i=k \\
q_i^{-1}t^{-1}-t^{-1}+1+q_i & \text{if } j<i<k \\
0 & \text{otherwise}

\end{array} \right.
\]
%\end{itemize}
\end{rmk}
%
%\begin{figure}[h]
%\begin{center}
%%\fontsize{10pt}{15pt}\selectfont% or whatever fontsize you like
%\def\svgwidth{0.9\columnwidth}
%\def\svgscale{0.3}
%\input{retprime.pdf_tex}
%\caption{$\delta_{1,1},\del-ta_{22}$ and $\delta_{2,2}$ \label{ret}}
%\end{center}
%\end{figure}

Similarly we compute:
\[
\langle N_1, F \rangle = - q_2 q_1^2
\]
\[
\langle N_4, F \rangle = q_4^{-1} t^{-1} .
\]

We detail the computation of the pairing of $F$ with $N_3$ (one can realize that it involves exactly the same paths as for $\langle N_i , F_{j,k} \rangle$ above with $j<i<k$). The situation is depicted in Figure \ref{FN3}.
\begin{figure}[h]
\begin{center}
%\fontsize{10pt}{15pt}\selectfont% or whatever fontsize you like
\def\svgwidth{0.7\columnwidth}
\def\svgscale{0.3}
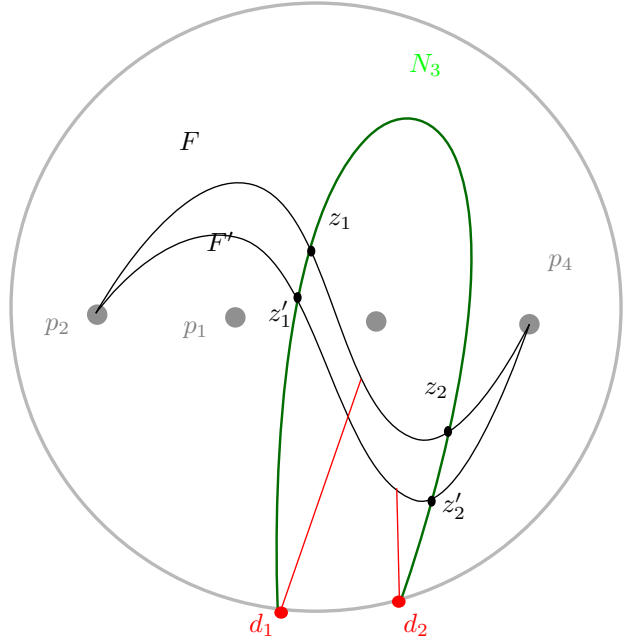
\caption{Intersection of fork $F$ with noodle $N_3$. \label{FN3}}
\end{center}
\end{figure}

$F$ and $N_3$ have two intersection points, the pairing involves four terms:
\begin{itemize}
\item for $\delta_{1,1}$ we get $m_{1,1} = q_3^{-1} t^{-1}$ so that $u_{1,1} = -1$ and that $\epsilon _{1,1} = 1$,
\item for $\delta_{2,2}$ we get $m_{2,2} = q_3$ so that $u_{2,2} = 0$ and that $\epsilon _{1,1} = -1$,
\item for $\delta_{1,2}$ we get $m_{1,2} = 1$ so that $u_{1,2} = 0$ and that $\epsilon _{1,2} = -(-1)^{u_{1,1}+u_{2,2}+u_{1,2}}=1$,
\item for $\delta_{1,2}$ we get $m_{2,1} = t^{-1}$ so that $u_{2,1} = -1$ and that $\epsilon _{2,1} = -(-1)^{u_{1,1}+u_{2,2}+u_{2,1}}=-1$.
\end{itemize}

Beside $\delta_{1,2}$ which is trivial, we draw $\delta_{1,1}$, $\delta_{2,2}$ and $\delta_{2,1}$ in Figure \ref{deltas} from which above computations are immediate. 
Finally:
\[
\langle N_3,F \rangle = q_3^{-1}t^{-1}-t^{-1}+1+q_3 .
\]

\begin{figure}[h]
\begin{center}
%\fontsize{10pt}{15pt}\selectfont% or whatever fontsize you like
\def\svgwidth{0.9\columnwidth}
\def\svgscale{0.3}
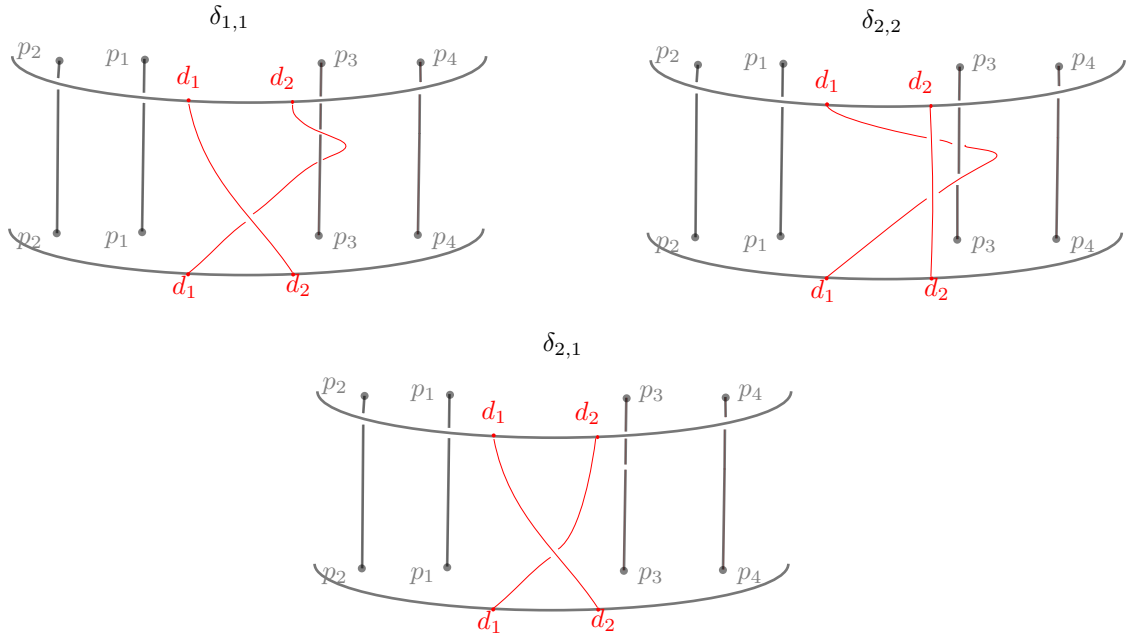
\caption{$\delta_{1,1},\delta_{22}$ and $\delta_{2,2}$ \label{deltas}}
\end{center}
\end{figure}

%\begin{figure}[h]
%\begin{center}
%%\fontsize{10pt}{15pt}\selectfont% or whatever fontsize you like
%\def\svgwidth{0.4\columnwidth}
%\def\svgscale{0.3}
%\input{delta22.pdf_tex}
%\caption{$\delta_{2,2}$ \label{delta22}}
%\end{center}
%\end{figure}
%
%\begin{figure}[h]
%\begin{center}
%%\fontsize{10pt}{15pt}\selectfont% or whatever fontsize you like
%\def\svgwidth{0.4\columnwidth}
%\def\svgscale{0.3}
%\input{delta21.pdf_tex}
%\caption{$\delta_{2,1}$ \label{delta21}}
%\end{center}
%\end{figure}

Replacing the computations above in the expression:
\[
\langle N_i, F \rangle = A\langle N_i, F^{\tau}_{1,2} \rangle+B\langle N_i, F^{\tau}_{1,4} \rangle+C\langle N_i, F^{\tau}_{2,4} \rangle
\]
with $i=1$ we get the condition:
\[
A+B = q_2^2
\]
and with $i=3$:
\[
B+C=1 .
\]
We need one more condition. We obtain it by pairing with the noodle $N_{2,3}$ defined as the noodle starting at $d_1$ and running around the punctures $p_2$ and $p_3$ before coming back to $d_2$ (see Figure \ref{N23}, noodle oriented from left to right). 

\begin{figure}[h]
\begin{center}
%\fontsize{10pt}{15pt}\selectfont% or whatever fontsize you like
\def\svgwidth{0.5\columnwidth}
\def\svgscale{0.3}
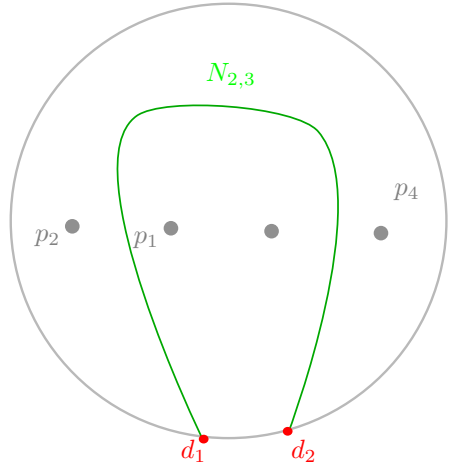
\caption{Noodle $N_{2,3}$. \label{N23}}
\end{center}
\end{figure}

We get the pairings:
\begin{itemize}
\item $\langle N_{2,3},F^{\tau}_{1,2} \rangle = (q_1 q_3)^{-1} t^{-1}$,
\item $\langle N_{2,3},F^{\tau}_{1,4} \rangle = (q_1 q_3)^{-1} t^{-1}-t^{-1}+1-q_1 q_3$,
\item $\langle N_{2,3},F^{\tau}_{2,4} \rangle = -q_1 q_3$,
\item $\langle N_{2,3},F \rangle = q_1 q_3^{-1} t^{-1}-q_1 t^{-1}+q_1-q_1 q_3=q_1 (1-q_3) (q_3^{-1} t^{-1}+1)$ .
\end{itemize}
By identification, we finally obtain:
\[
A= q_1^2-q_1 \text{ , } B= q_1 \text{ , } C= 1-q_21  .
\]
\end{example}
%\end{proof}

Proposition \ref{actioncBKL} allows computation of matrices. %The action described in the proposition is not multiplicative as the permutation induced by a braid shuffles the punctures and the corresponding variables in the action. In Chapter \ref{Unmodele}, Sections \ref{localsystemhomologytechniques} and \ref{homologyrelations}, we give homological tools that simplifies the computation of matrices and recovers the above proposition. 
In the above proposition we only provide the action on vectors $v \otimes 1$ for $v \in H_2(C)$. To deal with permuted vectors, one can transport punctures and permute variables consistently in the expressions of the proposition. As the only variable involved in the proposition is $q_i$, we end this section by a computational approach to these matrices, that indicates how to transport it. Let $BKL_i(q,t)$ be the matrix representing the action of $\sigma_i$ in the (uncolored) Bigelow-Krammer-Lawrence representation written in the basis $\{ F_{j,k} \}$ using the lexicographic order on $(i,k)$. See \cite{Big1} Section 4 for matrices, or use fonction {\em LKB\_matrix()} from the braid package of SageMath to obtain them. We consider the ones obtained from the above proposition with $q= q_1 = \cdots = q_n$, which corresponds to \cite[Section~4]{Big1}. It has entries in $\BZ \left[ q^{\pm 1} , t^{\pm 1} \right]$.  

Then it's a basic matrix computation that verifies the following remark.

\begin{rmk}\label{relationcBKL}
Let $q_1, \ldots, q_n$ be variables. Then:
\[
BKL_i(q_{i},t)BKL_j(q_{j},t)=BKL_j(q_{j},t)BKL_i(q_{i},t) \text{ for } |i-j| \geq 2
\]
\[
BKL_{i+1}(q_{i},t)BKL_i(q_{i},t)BKL_{i+1}(q_{i+1},t)= BKL_{i}(q_{i+1},t)BKL_{i+1}(q_{i},t)BKL_i(q_{i},t) .
\]
One can check this by a matrix computation. 
\end{rmk}

%\begin{rmk}
%Let $Q$ to be the variables ordered as follows:
%\[
%Q=(q_1,\ldots,q_n)
%\]
%then if $(k,k+1)$ designates the transposition between $k$'th and $(k+1)$'th element of a list, we have that:
%\[B_i((j,j+1)(Q),t)B_j(Q,t)=B_j((i,i+1)(Q),t)B_i(q_1,\cdots,q_n,t)
%\]
%\[
%B_{i+1}((i,i+1)(i+1,i+2)(Q),t)B_i((i+1,i+2)(Q),t)B_{i+1}(Q,t)=

%\\
%B_{i}((i+1,i+2)(i,i+1)(Q),t)B_{i+1}((i,i+1)(Q),t)B_i(Q,t)
%\]
%\end{rmk}

Now we can define the colored BKL matrix associated to a braid. 
\begin{defn}
Let $\alpha$ be a braid having the following word decomposition in the standard generators: $\alpha = \prod_{m=1}^{k} \sigma_{i_m}^{s_m}$ where $s_m$ are signs. Let $j_m$ be the index of the  `` over " strand at the $m$'th crossing in $\alpha$, braids read from right to left. Let the matrix $cBKL(\alpha)$ associated to the braid $\alpha$ be:
\begin{equation}\label{cBKLmatrix}
cBKL(\alpha) := \prod_{m=1}^{k} BKL_{i_m}(q_{j_m},t)^{s_m} .
\end{equation}
\end{defn}

Remark \ref{relationcBKL} shows that $cBKL$ is a well defined map between the braid group and the matrix group, but it is not multiplicative. 
%Before multiplying two braids we must know what strand is in what position which prevents, in the general case, from getting the matrix associated to a braid from the product of two matrices associated to two sub braids. Namely, if $\alpha= ab$ is a braid where $a$ and $b$ are two braids, $cBKL(\alpha)$ cannot be the product $cBKL(a)cBKL(b)$ unless $a$ brings the strands at their initial position, meaning $a$ is pure. 
For pure braids, $cBKL$ becomes a homomorphism and what we get is a representation of $\PBn$:
\[
cBKL: \left\lbrace \begin{array}{ccc}
\PBn & \to & GL_{{n}\choose{2}} \left(\BZ \left[ q_1^{\pm 1},\ldots , q_n^{\pm1},t \right] \right)
 \\
\alpha & \mapsto & cBKL(\alpha).
\end{array} \right.
\]

Remark \ref{relationcBKL} is a computational proof that this is a representation, i.e. that it satisfies braid relations. From Proposition \ref{actioncBKL} we remark that it is the colored BKL representation, corresponding to the initial homological definition (Proposition \ref{actioncBKL}). One remarks that the only variable involved in the action of $\sigma_i$ in Proposition \ref{actioncBKL} corresponds to the over passing strand. Specializing all $q_i$'s to a single variable $q$ recovers the uncolored BKL-representation. 

\begin{rmk}
In Section \ref{Gassnerrep}, we have presented a construction of the Gassner representation as a generalization of the Burau representation. Namely we used the standard Burau block of matrix but one has to use the variable $t_i$ if the strand  $i$ is passing above, i.e. the coloring follows strands. Here the conclusion is the same: the colored BKL representation uses the BKL standard block but with formal variables following the index of the strands (it is clear in Formula (\ref{cBKLmatrix})). 
\end{rmk}

\section{Higher Lawrence representations}\label{lapartiereine}

Section \ref{Generalcontext} is a discussion about bases for computations of Lawrence representations. It emphasizes the importance in the choice of the basis in order to respect the Laurent polynomials structure of coefficients, and it discusses the general relation between Lawrence representations and quantum representations of braid groups. Section \ref{matricesinthebasisofCS} contains Proposition \ref{cLawaction} computing the Lawrence representation in the basis of code sequences previously introduced. Section \ref{Computationwithfourstrands} provides details of the computation in the case of four punctures, which proves Proposition \ref{cLawaction}. Section \ref{Concretecase3strands} gives explicit matrices in the case of $3$ punctures and $m=2$, and how to recover matrices from previous Section \ref{cBKL}, in this precise case. 

\subsection{Panorama: Bases and quantum relations}\label{Generalcontext}

\subsubsection{Bases and ring of coefficients}

In the proof of Proposition \ref{actioncBKL} we used Bigelow's technique to find a basis ($\{ v_{j,k}\}$) of the homology by Fox calculus computations.  In Remark \ref{thisremark} we used Bigelow's argument saying that there exists a diagonal matrix sending this basis of the homology to the forks basis. It turns out that this matrix is not invertible in $\Laurent$ but only generically when variables are specialized to complex values, so that forks are not a basis of the homology as an $\Laurent$-module (we say that it is not an {\em integral basis}), see \cite[Proposition~7.2]{Jules1} or \cite{P-P}. In fact, in \cite{Big1} as in Proposition \ref{actioncBKL} we are computing the action of the braid group on the sub-module generated by forks. 

Bases (said ``integral'') of this Lawrence homological modules are given in \cite{Jules1}, two of them are provided one is the so called {\em code sequences} basis. By expressing the change of bases from forks to code sequences, one obtains the morphism relating the representation on the entire homology to the one restricted to the forks module. This morphism is not invertible in $\Laurent$ but under generic conditions that are explicitly given in \cite[Proposition~7.2]{Jules1}. The fact that the module of forks is a strict sub-module of the entire homology as an $\Laurent$-module was first shown in \cite{P-P}.

In \cite{Jules1}, colored Lawrence representations are defined for all levels of the grading, see \cite[Lemma~6.34]{Jules1}. Homological rules are provided, so that one can compute matrices for these representations using integral bases (so to get the representation on the entire homology, defined as an $\Laurent$-module). In the present paper we used a different approach to compute actions (in Sections \ref{Gassnertoquantum} with Gassner representations and \ref{cBKL} with BKL representations), the purpose was to emphasize the appearance of Fox differential calculus in the computation of these two first levels of Lawrence representations grading. Such Fox calculus technique for computing the homology seems hard to be generalized for higher Lawrence representations ($m>2$), although it would be very interesting to obtain such a generalization.

\subsubsection{Links with quantum representations}
In \cite{JK} the authors prove that the BKL representation is isomorphic to the restriction of the braid representation on $W_{n,2}$ (see \ref{GoodsubrepKJ}) to the space of {\em highest weight vectors} denoted $Y_{n,2}$. In \cite{Jules1} graded Lawrence representations of the braid groups (over modules denoted $\Habs_r$, for $r \in \BN$) are extended to graded relative homology modules (denoted $\Hrelm_r$). These extended representations are shown to be isomorphic to quantum representation of the braid group on $W_{n,r}$, \cite[Theorem~3]{Jules1}. It generalizes Theorem \ref{Gassnerarequantum} and the result from \cite{JK} for BKL representations.

\subsection{Matrices in the basis of code sequences}\label{matricesinthebasisofCS}

In the spirit of what we did in Section \ref{cBKL}, there exists a notion of multiforks to designate elements of $\Habs_m := H_m^{lf} (C_{n,m},L_m)$, see \cite{Ito2} or \cite[Section~7.1]{Jules1} for a colored version (Notation $\Habs_m$ is taken from \cite{Jules1} with index $m$ instead of $r$). In this section we suggest a computation of matrices for higher Lawrence representations ($m \ge 2$), in the colored version but in the basis of code sequences presented in \cite[Section~3.1]{Jules1}. It constitutes a basis of homology modules, see \cite[Lemma~3.1]{Big0} for a proof, or \cite[Proposition~3.1]{Jules1} which is an extended version of the latter. The morphism sending multiforks to code sequences is given in \cite[Proposition~7.2]{Jules1}. We recall the colored version for the local system $L_m$ used in this section, involving a choice of base point for $C_{n,m}$:
\[
L_m: \bfct
\BZ \lbrack \pi_1(C_{n,m},{\bf d}) \rbrack & \to & \Laurent:= \BZ\lbrack s_i^{\pm 1} , t^{\pm 1 } \rbrack_{i=1 , \ldots , n} \\ %:= \Laurent\\
{\pmb \xi} & \mapsto & \prod_i s_i^{w_i(\xi)} t^{u(\xi)}
\efct
\] 
where ${\bf d }  = \lbrace d_1 , \ldots , d_m \rbrace$ is a base point chosen so that coordinates lie on the boundary of the disk, $w_i$ is the invariant computing the total winding number of path ${\pmb \xi}$ around puncture $p_i$, $u$ is the winding number between configuration points, these invariants of loops are defined by analogy as in Section \ref{constructionandfaithfulness} in the case of $m=2$. For a more precise definition of this local system, using generators of the fundamental group, refer to \cite[Definition~2.4]{Jules1}, the correspondence between present variables and ones from \cite{Jules1} is $s_i = q^{-2\alpha_i}$ and $t=t$. In first subsection we present code sequences basis for $\Habs_m$ while in the second one we state propositions providing expressions for matrices of colored Lawrence representations in the basis of code sequences. In next subsection we will prove these propositions by performing a computation in the case of $4$ punctures.

\subsubsection{Code sequences}

\begin{defn}\label{Eznrdef}
We define the set of partitions of $n-1$ in $m$ integers as follows:
\[
\Enm = \lbrace (k_1, \ldots , k_{n-1}) \in \BN^{n-1} \text{ s.t. } \sum k_i=m  \rbrace .
\]
\end{defn}

We now define two families of topological objects indexed by $\Enm$, that will correspond to classes in $\Habs_m$.  

\begin{defn}[Code sequences diagrams]\label{codesequencesdiag}
We draw topological objects inside the punctured disk, the gray color is used to draw the punctured disk. Red arcs are going from a coordinate of the base point ${\bf d}$ of $C_{n,m}$ lying in its boundary to a dashed black arc. Dashed black arcs are oriented from left to right.  Finally the red arcs will end up going like in the following picture inside the dashed box, so that all families of red arcs are attached to the base point $\lbrace d_1, \ldots , d_m \rbrace$ of $C_{n,m}$ (here, $m'=m-k_1$), and that we will sometimes omit their ends in what follows when no confusion is possible.

\begin{equation*}
\begin{tikzpicture}[scale = 0.8]
%\node (w0) at (-5,1) {};
\node (w1) at (-1,1) {};
\node (w2) at (1,1) {};
\node[gray] at (2.0,1.0) {\ldots};
\node (wn1) at (3,1) {};
\node (wn) at (5,1) {};

%\node (x0) at (-4.8,-3) {};
%%\node at (-4.7,-3) {$\ldots$};
%\node  (x1) at (-4.3,-3) {};
%\node  (x2) at (-4.25,-3) {};
%\node  (xn1) at (,-3) {};
%\node  (xn) at (4,-3) {};

\draw[dashed,gray] (-5,0) -- (5,0) -- (5,-3);
\draw[gray] (-5,1.5) -- (-5,-3) -- (6,-3) node[right] {$\partial D_n$};

%\node[above,red] at (-3.5,0) {$k_1$};
\node[above,red] at (0,0) {$k_1$};
\node[above,red] at (4,0) {$k_{n-1}$};

%\draw[red] (-4,0.3) -- (-4,0) -- (-4.8,-3) node[below] {\small $d_m$};
%\draw[red] (-3,0.3) -- (-3,0) -- (-4.3,-3) node[below] {\small $d_{m'}$};
\draw[red] (-0.5,0.3) -- (-0.5,0) -- (-4.8,-3) node[below] {\small $d_m$};
\draw[red] (0.5,0.3) -- (0.5,0) -- (-4.3,-3) node[below] {\small $d_{m'}$};
\draw[red] (3.5,0.3) -- (3.5,0)-- (-3.2,-3) node[below] {\small $d_{k_{n-1}}$};
\draw[red] (4.5,0.3) -- (4.5,0)-- (-2.5,-3) node[below] {\small $d_1$};

%\node[red] at (-3.8,-1) {$\ldots$};
\node[red] at (-1.2,-0.9) {$\ldots$};
\node[red] at (0.25,-0.9) {$\ldots$};
\node[red] at (2.9,-0.5) {$\ldots$};

%\node[gray] at (w0)[left=5pt] {$p_0$};
\node[gray] at (w1)[above=5pt] {$p_1$};
\node[gray] at (w2)[above=5pt] {$p_2$};
\node[gray] at (wn1)[above=5pt] {$p_{n-1}$};
\node[gray] at (wn)[above=5pt] {$p_n$};
\foreach \n in {w1,w2,wn1,wn}
  \node at (\n)[gray,circle,fill,inner sep=3pt]{};
%\node at (w0)[gray,circle,fill,inner sep=3pt]{};
\end{tikzpicture}
\end{equation*}
%\begin{itemize}
%\item[{\bf Code sequences}] 
Let $ {\bf k} = (k_1 , \ldots , k_{n-1}) \in \Enm$, we define the {\em code sequence} $U_{\bf k}=U(k_1,\ldots,k_{n-1})$ to be the followingg drawing.
\begin{equation*}
\begin{tikzpicture}
%\node (w0) at (-5,0) {};
\node (w1) at (-3,0) {};
\node (w2) at (-1,0) {};
\node[gray] at (0.0,0.0) {\ldots};
\node (wn1) at (1,0) {};
\node (wn) at (3,0) {};
  
%\tikzstyle{grisEncadre}=[thick, dashed, fill=gray!20]
%\draw[dashed] (w0) -- (w1) node[midway, above] {$k_0$};
\draw[dashed] (w1) -- (w2) node[midway, above] {$k_1$};
\draw[dashed] (wn1) -- (wn) node[midway, above] {$k_{n-1}$};
%%\draw (-2,0) to (-1,0);
%
%\draw (1,0) to (2,0);
%%\draw (2,0) to (3,0);

%\node[gray] at (w0)[left=5pt] {$w_0$};
\node[gray] at (w1)[above=5pt] {$p_1$};
\node[gray] at (w2)[above=5pt] {$p_2$};
\node[gray] at (wn1)[above=5pt] {$p_{n-1}$};
\node[gray] at (wn)[above=5pt] {$p_n$};
\foreach \n in {w1,w2,wn1,wn}
  \node at (\n)[gray,circle,fill,inner sep=3pt]{};
%\node at (w0)[gray,circle,fill,inner sep=3pt]{};

%\draw[double,thick,red] (-4,-0.1) -- (-4,-2);
\draw[double,thick,red] (2,-0.1) -- (2,-2);
\draw[double,thick,red] (-2,-0.1) -- (-2,-2);

\draw[dashed,gray] (-5,-2) -- (3,-2);
\draw[dashed,gray] (3,-2) -- (3,-3);

\node[gray,circle,fill,inner sep=0.8pt] at (-4.8,-3) {};
\node[below,gray] at (-4.8,-3) {$d_m$};
\node[below=5pt,gray] at (-4.2,-3) {$\ldots$};
\node[gray,circle,fill,inner sep=0.8pt] at (-3.5,-3) {};
\node[below,gray] at (-3.5,-3) {$d_1$};

\draw[red] (-4.8,-3) -- (-2,-2);
\draw[red] (-3.5,-3) -- (2,-2);
\node[red] at (-2.3,-2.4) {$\ldots$};

\draw[gray] (-5,0) -- (-5,1);
\draw[gray] (-5,0) -- (-5,-3);
\draw[gray] (-5,-3) -- (4,-3) node[right] {$\partial D_n$};
\end{tikzpicture}
\end{equation*}
The indexes $k_i$'s stand to illustrate the fact that $k_i$ configuration points are embedded in the corresponding dashed segment, as we explain in what follows. We have attached to an indexed $k_i$ dashed arc a red arc called a {\em $(k_i)$-handle}. It is represented by a little red tube which is a simpler representation used to represent $k_i$ parallel red arcs that are called {\em handles}, in the spirit of what we did for forks. We let $\CU= \left\lbrace U(k_1, \ldots , k_{n-1}) \right\rbrace_{{\bf k} \in \Enm}$ designate the family of code sequences. The definition of these objects comes from \cite{Big0}. 
\end{defn} 

Now we explain how to assign a class in $\Habs_m$ from this diagram, by analogy of what we did for Forks in Section \ref{pairingbetweenforkandnoodles}. Let ${\bf k} \in \Enm$ and for all $i = 1 , \ldots ,n-1$, let:
\[
\phi_i : I_i \to D_n
\]
be the embedding of the dashed black arc number $i$ of $U(k_1 , \ldots , k_{n-1})$ indexed by $k_{i}$, where $I_i$ is a unit interval.
Let $\Delta^k$ be the standard (open) $k$ simplex:
\[
\Delta^k = \lbrace 0 < t_1 < \cdots < t_k < 1 \rbrace 
\]
for $k \in \BN$. It can also be thought as the configuration of $k$ points inside the unit interval so that, for all $i$, the map $\phi^{k_{i}}$:
\[
\phi^{k_{i}}: \bfct
\Delta^{k_{i}} & \to & C_{n,k_i} \\
(t_1, \ldots , t_{k_{i}} ) & \mapsto & \lbrace \phi_i(t_1) , \ldots, \phi_i(t_{k_{i}}) \rbrace
\efct
\]
is a well defined map. It is a locally finite cycle. These two last facts are detailed in \cite[Section~3.1]{Jules1}. We use handles to get a class in the local system homology as we did for forks. To get a cycle in the local system homology, one has to choose a lift of the chain to the maximal abelian cover of $C_{n,m}$ associated with the morphism $L_m$. The way to do so is using the red handles to which is canonically associated a path:
\[
{\bf h}=\lbrace h_1,\ldots,h_m \rbrace: I \to C_{n,m}
\]
joining the base point ${\bf d}$ and the $m$-chain assigned to dashed arcs. At the cover level there is a unique lift $\widehat{{\bf h}}$ of ${\bf h}$ that starts at $\widehat{{\bf d}}$. The lift of $U(k_1, \ldots , k_{n-1})$ passing by $\widehat{\bf h} (1)$ defines a cycle, so a class in $\Habs_m$, that we still call $U(k_1 , \ldots , k_{n-1})$.

The above construction of class is made in \cite[Section~3.1]{Jules1}, but in some relative homology case, which involves adding one dashed arc going to the boundary. The following is proved in \cite[Lemma~3.1]{Big0}, and rephrased in the relative homology formalism in \cite[Proposition~3.1]{Jules1}. 
\begin{prop}
The group $\Habs_m$ is a free $\Laurent$-module for which the family $\CU$ is a basis. 
\end{prop}

\subsubsection{Matrices for colored Lawrence representations}

We compute the action of the braid group $\Bn$ on $\Habs_m$ in the basis of code sequences. Actually, the order of punctures $p_i$'s is of importance as they can be permuted by braids. We designate by ${\Habs_m}^{()}$ the space built from $C_{n,m}$ with punctures ordered from $1$ to $n$ ( $()$ refers to the identity permutation). For $\tau \in \Sk_n$, we designate by ${\Habs_m}^{\tau}$ the one obtained from $C_{n,m}^{\tau}$ with punctures permuted by $\tau$. As for previous part, the action of $\Bn$ is over $\bigoplus_{\tau \in \Sk_n} {\Habs_m}^{\tau}$, see the appendix in Section \ref{pureandcolored}. We need new quantum numbers to perform homology computations.

\begin{defn}\label{quantumt}
Let $i$ be a positive integer. We define the following elements of $\BZ \left[ t^{\pm 1} \right] \subset \Laurent$.
%\[
%\begin{array}{ccc}
\begin{equation*}
(i)_t := (1+t+ \cdots + t^{i-1}) = \frac{1-t^i}{1-t} , \text{  } (k)_t! := \prod_{i=1}^k (i)_t, \text{ and }  {{k}\choose{l}}_t := \frac{(k)_t!}{(k-l)_t! (l)_t!} %= \frac{(t,t)_k}{(t,t)_l (t,t)_{k-l}} .
%\end{array}
%\]
\end{equation*}
We also define quantum trinomials as follows:
\[
{{n}\choose{i,j,k}}_t = \frac{(n)_t}{(i)_t (j)_t (k)_t}. 
\]
\end{defn}

\begin{prop}[Colored Lawrence action]\label{cLawaction}
Let $n \in \BN$, and $(k_1 , \ldots , k_{n-1}) \in \Enm$. The action of standard generators of $\Bn$ on ${\Habs_m}^{()}$ is computed on a standard code sequence as follows:
\[
\sigma_1 \cdot U^{()}(k_1 , \ldots , k_{n-1}) = (-1)^{k_1} t^{-\frac{k_{1}\left(k_{1}-1\right)}{2}} \sum_{l=0}^{k_2} s_1^{k_1+l} {{k_1+l}\choose{k_1}}_{t^{-1}} U^{(1,2)}(k_1+l,k_2-l,k_3,\ldots , k_{n-1}), \]\[
\sigma_{n-1} \cdot U^{()}(k_1,\ldots , k_{n-1}) = (-1)^{k_{n-1}} t^{-\frac{k_{n-1}\left(k_{n-1}-1\right)}{2}} s_{n-1}^{k_{n-1}} \sum_{l=0}^{k_{n-2}} {{k_{n-1}+l}\choose{k_{n-1}}}_{t^{-1}} U^{(n-2,n-1)}(k_1,\ldots,k_{n-2}-l,k_{n-1}+l),
\]\[
\sigma_i \cdot U^{()}(k_1 , \ldots , k_{n-1}) = (-1)^{k_i} t^{-\frac{k_i\left(k_i-1\right)}{2}}  \sum_{l_1=0}^{k_{i-1}} \sum_{l_2=0}^{k_{i+1}}  s_i^{k_i+l_2} {{k_i+l_1+l_2}\choose{k_i,l_1,l_2}}_{t^{-1}}  U^{(i,i+1)}_{i;l_1,l_2} %U^{(i,i+1)}(k_1,\ldots k_{i-1} - l_1 , k_i+ l_1+ l_2, k_{i+1} - l_2 , \ldots , k_{n-1}),
 \]
for $i = 2 , \ldots , n-2$, where:
\[
U^{(i,i+1)}_{i;l_1,l_2} = U^{(i,i+1)}(k_1,\ldots, k_{i-1} - l_1 , k_i+ l_1+ l_2, k_{i+1} - l_2 , \ldots , k_{n-1}).
\]
\end{prop}
\begin{proof}
As a half Dehn twist is a ``local move'' in the sense that it only involves arcs reaching the swapped two points, the computation of generators' action in all cases of punctures, is a straightforward consequence of the $4$ punctured case. We perform this computation in Example \ref{mainexample} of next section, from which it suffices to replace $2$ by $i$ in the action of $\sigma_2$ to obtain that of $\sigma_i$, and to replace $3$ by $n-1$ in the action of $\sigma_3$ to deduce that of $\sigma_{n-1}$.  
\end{proof}

The above proposition is sufficient to get matrices for colored Lawrence representations by considering the following remark. 

\begin{rmk}
The above proposition compute representation of a braid $\beta \in \Bn$ as an element of $$\Hom_{\Laurent} \left( {\Habs_m}^{()} , {\Habs_m}^{\perm(\beta)} \right).$$
To be able to write matrices for the representation of $\Bn$ on $\bigoplus_{\tau \in \Sk_n} {\Habs_m}^{\tau}$ one has to compute the action of braids on elements $U^{\tau}$ for $\tau \in \Sk_n$ instead of $U^{()}$. To do so, one has to take formulas from Proposition \ref{cLawaction} transporting variables $s_i$'s by $\tau^{-1}$, namely replacing $s_i$ by $s_{\tau^{-1}(i)}$ for all $i = 1 , \ldots , n$. 
\end{rmk}

Matrices for the (uncolored) Lawrence representation of braid groups is an immediate corollary of Proposition \ref{cLawaction}, by equalizing variables $s_i$'s to a single one. 

\begin{coro}[(uncolored) Lawrence action]\label{Lawaction}
Let $n \in \BN$, $s:=s_1=\cdots=s_n$, and $(k_1 , \ldots , k_{n-1}) \in \Enm$. The representation of $\Bn$ on ${\Habs_m}$ is given by the action of its generators on the standard code sequences basis as follows:
\[
\sigma_1 \cdot U(k_1 , \ldots , k_{n-1}) = (-1)^{k_1} t^{-\frac{k_1\left(k_1-1\right)}{2}} \sum_{l=0}^{k_2} s^{k_1+l} {{k_1+l}\choose{k_1}}_{t^{-1}} U(k_1+l,k_2-l,k_3,\ldots , k_{n-1}), \]\[
\sigma_{n-1} \cdot U(k_1,\ldots , k_{n-1}) = (-1)^{k_{n-1}} t^{-\frac{k_{n-1}\left(k_{n-1}-1\right)}{2}} s^{k_{n-1}} \sum_{l=0}^{k_{n-2}} {{k_{n-1}+l}\choose{k_{n-1}}}_{t^{-1}} U(k_1,\ldots,k_{n-2}-l,k_{n-1}+l),
\]\[
\sigma_i \cdot U(k_1 , \ldots , k_{n-1}) = (-1)^{k_i} t^{-\frac{k_i\left(k_i-1 \right)}{2}}  \sum_{l_1=0}^{k_{i-1}} \sum_{l_2=0}^{k_{i+1}}  s^{k_i+l_2} {{k_i+l_1+l_2}\choose{k_i,l_1,l_2}}_{t^{-1}}  U_{i;l_1,l_2} %U^{(i,i+1)}(k_1,\ldots k_{i-1} - l_1 , k_i+ l_1+ l_2, k_{i+1} - l_2 , \ldots , k_{n-1}),
 \]
for $i = 2 , \ldots , n-2$, where:
\[
U^{(i,i+1)}_{i;l_1,l_2} = U^{(i,i+1)}(k_1,\ldots, k_{i-1} - l_1 , k_i+ l_1+ l_2, k_{i+1} - l_2 , \ldots , k_{n-1}).
\]
\end{coro}

Then the following remark tells one how to compute matrices for the colored version (as a representation of the pure braid group) out of matrices of Lawrence representation given in the above corollary.

\begin{rmk}[Coloring the Lawrence representation]\label{coloringLaw}
Let $L_i(s)$ be the matrix associated with $\sigma_i \in \Bn$ by the uncolored Lawrence representation over ${\Habs_m}$ written in the code sequence basis (given in the above Corollary \ref{Lawaction}). Let $\beta \in \Bn$ such that:
\[
\beta = \prod_{m=1}^{k} \sigma_{i_m}^{\epsilon_m}
\]
where $\epsilon_m$ are signs ($\pm 1$). Let $j_m$ be the index of the  ``over'' passing strand at the $m$'th crossing of $\beta$, braids read from right to left. 
Then:
\[
cL(\beta) := \prod_{m=1}^{k} L_{i_m}(s_{j_m})^{\epsilon_m}
\]
is a well defined matrix associated with $\beta$. For pure braids, it provides the colored representation of $\PBn$ on ${\Habs_m}^{()}$
\end{rmk}

\begin{rmk}
The matrices provided in this section work also for the case $m=2$, and provide matrices for the BKL representations. In this case there is a change of basis relating them to those from Proposition \ref{actioncBKL}.  This is detailed in Section \ref{Concretecase3strands} below in the case of the three strands braid group, which should allow the reader to understand the general case.%One example of this relation is provided in Section \ref{Concretecase3strands} below, which should allow one to understand the general case. 
\end{rmk}

\subsection{Computation with four strands}\label{Computationwithfourstrands}

We compute the action of generators in the case of $\CB_4$. We use homological techniques developed in \cite[Section~4]{Jules1}, not reproving them. Still we give subtle examples of {\em handle rules} computation at the end of present section, in Example \ref{handleexamples}.

\subsubsection{Computation}

\begin{example}[Computation of $\CB_4$ representations]\label{mainexample}
Let $n=4$. We compute the action of $\sigma_2 \in \CB_4$ on the code sequence $U:=U(k_1,k_2,k_3)$ such that $\sum k_i = m$. 

\begin{align*}
\sigma_2 \cdot U & = (-1)^{k_2} \left(\vcenter{\hbox{
\begin{tikzpicture}[decoration={
    markings,
    mark=at position 0.5 with {\arrow{>}}}
    ]
\node (w1) at (-3,0) {};
\node (w2) at (-1,0) {};
\coordinate (a) at (-3.2,-2);
%\node[gray] at (0.0,0.0) {\ldots};
\node (wn1) at (1,0) {};
\node (wn) at (3,0) {};  
\draw[dashed] (w1) to[bend right=30] node[midway, below] {$k_1$} node[pos=0.4,above] (k1) {} (wn1);
\draw[dashed] (w2)to  node[pos=0.4] {$k_2$} node[pos=0.7] (k2) {} (wn1);
\draw[dashed] (w2) to[bend left=30] node[midway, above] {$k_3$} node[pos=0.8,above] (k3) {} (wn);
\node[gray] at (w1)[above=5pt] {$p_1$};
\node[gray] at (w2)[above=5pt] {$p_3$};
%\node[gray] at (wn1)[above=5pt] {$p_2$};
\node[gray] at (wn)[above=5pt] {$p_4$};
\foreach \n in {w1,w2,wn1,wn}
  \node at (\n)[gray,circle,fill,inner sep=3pt]{};  
\draw[double,red,thick] (k1) -- (k1|-a);
\draw[double,red,thick] (k2) arc(180:0:0.5) node[above] (k2p) {};
\draw[double,red,thick] (k2p) -- (k2p|-a);
\draw[double,red,thick] (k3) -- (k3|-a);
\end{tikzpicture}
}} \right) \\
& = (-1)^{k_2} s_2^{k_2} t^{-\frac{k_2(k_2-1)}{2}} \left(\vcenter{\hbox{
\begin{tikzpicture}[decoration={
    markings,
    mark=at position 0.5 with {\arrow{>}}}
    ]
\node (w1) at (-3,0) {};
\node (w2) at (-1,0) {};
\coordinate (a) at (-3.2,-2);
%\node[gray] at (0.0,0.0) {\ldots};
\node (wn1) at (1,0) {};
\node (wn) at (3,0) {};  
\draw[dashed] (w1) to[bend right=30] node[midway, below] {$k_1$} node[pos=0.4,above] (k1) {} (wn1);
\draw[dashed] (w2)to  node[midway, above] {$k_2$} node[midway,above] (k2) {} (wn1);
\draw[dashed] (w2) to[bend left=30] node[midway, above] {$k_3$} node[pos=0.7,above] (k3) {} (wn);
\node[gray] at (w1)[above=5pt] {$p_1$};
\node[gray] at (w2)[above=5pt] {$p_3$};
\node[gray] at (wn1)[above=5pt] {$p_2$};
\node[gray] at (wn)[above=5pt] {$p_4$};
\foreach \n in {w1,w2,wn1,wn}
  \node at (\n)[gray,circle,fill,inner sep=3pt]{};  
\draw[double,red,thick] (k1) -- (k1|-a);
\draw[double,red,thick] (k2) -- (k2|-a);
\draw[double,red,thick] (k3) -- (k3|-a);
\end{tikzpicture}
}} \right) 
\end{align*}

The coefficient $(-1)^{k_2}$ shows up for reversing the orientation of the indexed $k_2$ dashed arc (that has been reversed by the half Dehn twist $\sigma_2$) so that all dashed arcs are still oriented from left to right. The coefficient $s_2^{k_2}t^{-\frac{k_2(k_2-1)}{2}}$ stands for preserving a straight $(k_2)$-handle: after the application of $\sigma_2$, the $(k_2)$-handle runs once around puncture $p_2$, we use the {\em handle rule} introduced in \cite[Remark~4.2]{Jules1}, and recalled in Remark \ref{handleruless}, that precises the coefficient appearing while modifying a handle. This precise handle rule coefficient is given in below Example \ref{handleexamples} (i).

\begin{align*}
\sigma_2 \cdot U(k_1,k_2,k_3) & = (-1)^{k_2} s_2^{k_2} t^{-\frac{k_2(k_2-1)}{2}} \sum_{l_1=0}^{k_1} \sum_{l_2=0}^{k_3} t^{\frac{l_2(l_2-1)}{2}} \left(\vcenter{\hbox{
\begin{tikzpicture}[decoration={
    markings,
    mark=at position 0.5 with {\arrow{>}}}
    ]
\node (w1) at (-3,0) {};
\node (w2) at (-1,0) {};
\coordinate (a) at (-3.2,-2);
%\node[gray] at (0.0,0.0) {\ldots};
\node (wn1) at (1,0) {};
\node (wn) at (3,0) {};  
\draw[dashed] (w1) to node[midway, above] {$k_1-l_1$} node[midway,above] (k1) {} (w2);
\draw[dashed] (w2) to node[midway, above] {$k_2$} node[midway,above] (k2) {} (wn1);
\draw[dashed] (w2) to[bend right=50] node[pos=0.4, above] {$l_1$} node[pos=0.3,above] (l1) {} (wn1);
\draw[dashed] (w2) to[bend left=50] node[midway, above] {$l_2$} node[pos=0.7] (l2) {} (wn1);
\draw[dashed] (wn1) to node[midway, above] {$k_3-l_2$} node[pos=0.5,above] (k3) {} (wn);
\node[gray] at (w1)[above=5pt] {$p_1$};
\node[gray] at (w2)[above=5pt] {$p_3$};
\node[gray] at (wn1)[above=5pt] {$p_2$};
\node[gray] at (wn)[above=5pt] {$p_4$};
\foreach \n in {w1,w2,wn1,wn}
  \node at (\n)[gray,circle,fill,inner sep=3pt]{};  
\draw[double,red,thick] (k1) -- (k1|-a);
\draw[double,red,thick] (k2) -- (k2|-a);
\draw[double,red,thick] (l2) arc(180:0:0.5) node[above] (int) {};
\draw[double,red,thick] (int) -- (int|-a);
\draw[double,red,thick] (k3) -- (k3|-a);
\draw[double,red,thick] (l1) -- (l1|-a);
%\draw[double,red,thick] (l2) -- (l2|-a);
\end{tikzpicture} }} \right) \\
& = (-1)^{k_2} t^{-\frac{k_2(k_2-1)}{2}}  \sum_{l_1=0}^{k_1} \sum_{l_2=0}^{k_3} s_2^{k_2+l_2} \left(\vcenter{\hbox{
\begin{tikzpicture}[decoration={
    markings,
    mark=at position 0.5 with {\arrow{>}}}
    ]
\node (w1) at (-3,0) {};
\node (w2) at (-1,0) {};
\coordinate (a) at (-3.2,-2);
%\node[gray] at (0.0,0.0) {\ldots};
\node (wn1) at (1,0) {};
\node (wn) at (3,0) {};  
\draw[dashed] (w1) to node[midway, above] {$k_1-l_1$} node[midway,above] (k1) {} (w2);
\draw[dashed] (w2) to node[midway, above] {$k_2$} node[midway,above] (k2) {} (wn1);
\draw[dashed] (w2) to[bend right=50] node[pos=0.4, above] {$l_1$} node[pos=0.3,above] (l1) {} (wn1);
\draw[dashed] (w2) to[bend left=50] node[midway, above] {$l_2$} node[pos=0.7,above] (l2) {} (wn1);
\draw[dashed] (wn1) to node[midway, above] {$k_3-l_2$} node[pos=0.5,above] (k3) {} (wn);
\node[gray] at (w1)[above=5pt] {$p_1$};
\node[gray] at (w2)[above=5pt] {$p_3$};
\node[gray] at (wn1)[above=5pt] {$p_2$};
\node[gray] at (wn)[above=5pt] {$p_4$};
\foreach \n in {w1,w2,wn1,wn}
  \node at (\n)[gray,circle,fill,inner sep=3pt]{};  
\draw[double,red,thick] (k1) -- (k1|-a);
\draw[double,red,thick] (k2) -- (k2|-a);
\draw[double,red,thick] (k3) -- (k3|-a);
\draw[double,red,thick] (l1) -- (l1|-a);
\draw[double,red,thick] (l2) -- (l2|-a);
\end{tikzpicture} }} \right)
\end{align*}
The first equality comes from successive breaking of dashed arcs, which is a diagram rule presented in \cite[Example~4.6]{Jules1}, of indexed $(k_1)$ and $(k_3)$ dashed arcs respectively, with a detailed utilization of it presented in Example \ref{handleexamples} (ii). A coefficient $t^{\frac{l_2(l_2-1)}{2}}$ appears, see Example \ref{handleexamples} (ii). The second equality brings a coefficient $s_2^{l_2}t^{\frac{-l_2(l_2-1)}{2}}$ from the handle rule of Example \ref{handleexamples} (i). Lastly, we use successively two {\em fusions of dashed arcs}, presented in \cite[Corollary~4.10]{Jules1}, to transform middle theta (dashed) diagram into one dashed arc and to recover a code sequence, with the apparition of quantum trinomials in exchange. 
\begin{equation}\label{actionUn=4}
\sigma_2 \cdot U^{()}(k_1,k_2,k_3) = (-1)^{k_2} t^{-\frac{k_2(k_2-1)}{2}} \sum_{l_1=0}^{k_1} \sum_{l_2=0}^{k_3}  s_2^{k_2+l_2} {{k_2+l_1+l_2}\choose{k_2,l_1,l_2}}_{t^{-1}} U^{(2,3)}(k_1-l_1,k_2+ l_1+ l_2, k_3 - l_2)
\end{equation}
for $(2,3) \in \Sk_4$, where we used notations $U^{()}$ and $U^{\tau}$ to distinguish elements in ${\Habs_m}^{()}$ and ${\Habs_m}^{\tau}$ respectively. By using same part of computation (but simpler as we now deal with leftmost and rightmost generators), one can compute the action of the two other generators to obtain the following formulas.

\begin{align}\label{extremegenUn=4}
\sigma_1 \cdot U^{()}(k_1,k_2,k_3) = (-1)^{k_1} t^{-\frac{k_1(k_1-1)}{2}} \sum_{l=0}^{k_2} s_1^{k_1+l} {{k_1+l}\choose{k_1}}_{t^{-1}} U^{(1,2)}(k_1+l,k_2-l,k_3), \\
\sigma_3 \cdot U^{()}(k_1,k_2,k_3) = (-1)^{k_3} s_3^{k_3} t^{-\frac{k_3(k_3-1)}{2}} \sum_{l=0}^{k_2} {{k_3+l}\choose{k_3}}_{t^{-1}} U^{(3,4)}(k_1,k_2-l,k_3+l).
\end{align}
\end{example}

\subsubsection{Handle rule}

We recall the handle rule, and we give two subtle examples used several times in the previous computation.

\begin{rmk}[Handle rule, {\cite[Remark~4.1]{Jules1}}]\label{handleruless}
Let $B$ be a singular locally finite $r$-cycle of $C_m(C_{n,m},\BZ)$. We've seen a process to choose a lift of $B$ to the homology with local coefficients in $L_m$, using a handle which is a path joining ${\bf d}$ to $x \in B$. Let $\alpha$ and $\beta$ be two different paths joining ${\bf d}$ and $B$. Let $\widehat{B}^\alpha$ and $\widehat{B}^\beta$ be the lifts of $B$ chosen using $\alpha$ and $\beta$ respectively. By the {\em handle rule} we have the following relation in $\Habs_m$:
\begin{equation*}
\widehat{B}^\alpha = L_m(\beta \alpha^{-1}) \widehat{B}^\beta
\end{equation*}
where $L_m$ is the representation of $\pi_1(C_{n,m})$ used for the local system. This expresses how the local system coordinate of a homological class is translated after a change of handle. 
\end{rmk}

\begin{example}[Examples of handle rules]\label{handleexamples}
We provide two examples of handle rules applications in the case of a modification of a $(k)$-handle, standing for $k$ parallel handles.
\begin{itemize}
\item[(i)] A fist application of the handle rule is the following:
\[
\left(\vcenter{\hbox{ \begin{tikzpicture}[decoration={
    markings,
    mark=at position 0.5 with {\arrow{>}}}
    ]
\coordinate (w0) at (-1,0) {};
\coordinate (w1) at (1,0) {};
\coordinate (x0) at (-1,-1) {};
\coordinate (x1) at (1,-1) {};

%\draw[postaction={decorate}] (w0) -- node[above,pos=0.3] (k0) {} (w1);
\draw[dashed] (w0) to node[pos=0.6] (k1) {} node[pos=0.4,below] {$k$} (w1);

%\draw[red] (k0) -- node[midway] (a) {} (k0|-x0);
\draw[red,double,thick] (k1) arc(180:0:0.75) node[above] (k2) {};
\draw[red,double,thick] (k2) -- (k2|-x1);

%\node[right,red] at (a) {$\alpha$};

\node[gray] at (w0)[left=5pt] {$p_i$};
\node[gray] at (w1)[right=5pt] {$p_j$};
\foreach \n in {w0,w1}
  \node at (\n)[gray,circle,fill,inner sep=3pt]{};

\end{tikzpicture} }}\right)
= s_j^{k} t^{\frac{-k(k-1)}{2}}
\left(\vcenter{\hbox{ \begin{tikzpicture}[decoration={
    markings,
    mark=at position 0.5 with {\arrow{>}}}
    ]
\coordinate (w0) at (-1,0) {};
\coordinate (w1) at (1,0) {};
\coordinate (x0) at (-1,-1) {};
\coordinate (x1) at (1,-1) {};

%\draw[postaction={decorate}] (w0) -- node[above,pos=0.3] (k0) {} (w1);
\draw[dashed] (w0) to node[pos=0.5,above] (k1) {} node[pos=0.5,above] {$k$} (w1);

%\draw[red] (k0) -- node[midway] (a) {} (k0|-x0);
%\draw[red,double] (k1) arc(180:0:0.75) node[above] (k2) {};
\draw[red,double,thick] (k1) -- (k1|-x1);

%\node[right,red] at (a) {$\alpha$};

\node[gray] at (w0)[left=5pt] {$p_i$};
\node[gray] at (w1)[right=5pt] {$p_j$};
\foreach \n in {w0,w1}
  \node at (\n)[gray,circle,fill,inner sep=3pt]{};

\end{tikzpicture} }}\right)
\]
where drawings are the same outside parenthesis. This is a generalization of \cite[Example~4.2]{Jules1}. One can deduce the above coefficient from it, by replacing the simple handle $\beta$ by a $(k)$-handle that is a ribbon of $k$ parallel handles. By replacing a simple strand (see \cite[Figure~3]{Jules1}) by a ribbon of $(k)$-strands running once around puncture $p_j$, there is a coefficient $s_j^{k}$ appearing for the total winding number around $p_j$. There is also a coefficient $t^{\frac{-k(k-1)}{2}}$ appearing for the twist of the ribbon necessary for the ribbon to encircle the puncture. Twisting the ribbon is assimilated to a framed Reidemeister I move, involving a ribbon effect.  
\item[(ii)] As in (i), we have:
\[
\left(\vcenter{\hbox{ \begin{tikzpicture}[decoration={
    markings,
    mark=at position 0.5 with {\arrow{>}}}
    ]
\coordinate (w0) at (-1,0) {};
\coordinate (w05) at (0.5,0) {};
\coordinate (w1) at (2,0) {};
\coordinate (x0) at (-1,-1) {};
\coordinate (x1) at (1,-1) {};

%\draw[postaction={decorate}] (w0) -- node[above,pos=0.3] (k0) {} (w1);
\draw[dashed] (w0) to[bend left=40] node[pos=0.7,above] (k1) {} node[pos=0.5,above] {$k$} (w1);

%\draw[red] (k0) -- node[midway] (a) {} (k0|-x0);
%\draw[red,double] (k1) arc(180:0:0.75) node[above] (k2) {};
\draw[red,double,thick] (k1) -- (k1|-x1);

%\node[right,red] at (a) {$\alpha$};

\node[gray] at (w0)[below] {$p_{i-1}$};
\node[gray] at (w05)[below] {$p_i$};
\node[gray] at (w1)[below] {$p_{i+1}$};
\foreach \n in {w0,w1,w05}
  \node at (\n)[gray,circle,fill,inner sep=3pt]{};

\end{tikzpicture} }}\right)
= t^{\frac{k(k-1)}{2}}
\left(\vcenter{\hbox{ \begin{tikzpicture}[decoration={
    markings,
    mark=at position 0.5 with {\arrow{>}}}
    ]
\coordinate (w0) at (-1,0) {};
\coordinate (w05) at (0.5,0) {};
\coordinate (w1) at (2,0) {};
\coordinate (x0) at (-1,-1) {};
\coordinate (x1) at (1,-1) {};

%\draw[postaction={decorate}] (w0) -- node[above,pos=0.3] (k0) {} (w1);
\draw[dashed] (w0) to[bend left=40] node[pos=0.5] (k1) {} node[pos=0.3,above] {$k$} (w1);

%\draw[red] (k0) -- node[midway] (a) {} (k0|-x0);
\draw[red,double,thick] (k1) arc(180:0:0.4) node[above] (k2) {};
\draw[red,double,thick] (k2) -- (k2|-x1);

%\node[right,red] at (a) {$\alpha$};

\node[gray] at (w0)[below] {$p_{i-1}$};
\node[gray] at (w05)[below] {$p_i$};
\node[gray] at (w1)[below] {$p_{i+1}$};
\foreach \n in {w0,w1,w05}
  \node at (\n)[gray,circle,fill,inner sep=3pt]{};

\end{tikzpicture} }}\right)
\]
with the coefficient $t^{\frac{k(k-1)}{2}}$ coming from twisting the ribbon of $k$ parallel handles. So that:
\begin{align*}
\left(\vcenter{\hbox{ \begin{tikzpicture}[decoration={
    markings,
    mark=at position 0.5 with {\arrow{>}}}
    ]
\coordinate (w0) at (-1,0) {};
\coordinate (w05) at (0.5,0) {};
\coordinate (w1) at (2,0) {};
\coordinate (x0) at (-1,-1) {};
\coordinate (x1) at (1,-1) {};
\draw[dashed] (w0) to[bend left=40] node[pos=0.7,above] (k1) {} node[pos=0.5,above] {$k$} (w1);
\draw[red,double,thick] (k1) -- (k1|-x1);
\node[gray] at (w0)[below] {$p_{i-1}$};
\node[gray] at (w05)[below] {$p_i$};
\node[gray] at (w1)[below] {$p_{i+1}$};
\foreach \n in {w0,w1,w05}
  \node at (\n)[gray,circle,fill,inner sep=3pt]{};
\end{tikzpicture} }}\right)
& = \sum_{l=0}^{k} t^{\frac{l(l-1)}{2}}
\left(\vcenter{\hbox{ \begin{tikzpicture}[decoration={
    markings,
    mark=at position 0.5 with {\arrow{>}}}
    ]
\coordinate (w0) at (-1,0) {};
\coordinate (w05) at (0.5,0) {};
\coordinate (w1) at (2.3,0) {};
\coordinate (x0) at (-1,-1) {};
\coordinate (x1) at (1,-1) {};
\draw[dashed] (w0) to node[pos=0.7] (k1) {} node[pos=0.5,above] {$l$} (w05);
\draw[dashed] (w05) to node[pos=0.5,above] (k3) {} node[pos=0.6,above] {$k-l$} (w1);
\draw[red,double,thick] (k1) arc(180:0:0.4) node[above] (k2) {};
\draw[red,double,thick] (k2) -- (k2|-x1);
\draw[red,double,thick] (k3) -- (k3|-x1);
\node[gray] at (w0)[below] {$p_{i-1}$};
\node[gray] at (w05)[below] {$p_i$};
\node[gray] at (w1)[below] {$p_{i+1}$};
\foreach \n in {w0,w1,w05}
  \node at (\n)[gray,circle,fill,inner sep=3pt]{};
\end{tikzpicture} }}\right) \\
& = 
\sum_{l=0}^{k} s_i^{l}
\left(\vcenter{\hbox{ \begin{tikzpicture}[decoration={
    markings,
    mark=at position 0.5 with {\arrow{>}}}
    ]
\coordinate (w0) at (-1,0) {};
\coordinate (w05) at (0.5,0) {};
\coordinate (w1) at (2.3,0) {};
\coordinate (x0) at (-1,-1) {};
\coordinate (x1) at (1,-1) {};
\draw[dashed] (w0) to node[pos=0.5,above] (k1) {} node[pos=0.5,above] {$l$} (w05);
\draw[dashed] (w05) to node[pos=0.5,above] (k3) {} node[pos=0.6,above] {$k-l$} (w1);
%\draw[red,double,thick] (k1) arc(180:0:0.4) node[above] (k2) {};
\draw[red,double,thick] (k1) -- (k1|-x1);
\draw[red,double,thick] (k3) -- (k3|-x1);
\node[gray] at (w0)[below] {$p_{i-1}$};
\node[gray] at (w05)[below] {$p_i$};
\node[gray] at (w1)[below] {$p_{i+1}$};
\foreach \n in {w0,w1,w05}
  \node at (\n)[gray,circle,fill,inner sep=3pt]{};
\end{tikzpicture} }}\right) .
\end{align*}
The first equality comes from {\em breaking a dashed arc}, see \cite[Example~4.6]{Jules1}. The second one is an application of (i), so that powers of $t$ are simplified.

%= s_j^{k} t^{\frac{-k(k-1)}{2}}
%\left(\vcenter{\hbox{ \begin{tikzpicture}[decoration={
%    markings,
%    mark=at position 0.5 with {\arrow{>}}}
%    ]
%\coordinate (w0) at (-1,0) {};
%\coordinate (w1) at (1,0) {};
%\coordinate (x0) at (-1,-1) {};
%\coordinate (x1) at (1,-1) {};
%
%
%  
%
%%\draw[postaction={decorate}] (w0) -- node[above,pos=0.3] (k0) {} (w1);
%\draw[dashed] (w0) to node[pos=0.5,above] (k1) {} node[pos=0.5,above] {$k$} (w1);
%
%
%
%
%%\draw[red] (k0) -- node[midway] (a) {} (k0|-x0);
%%\draw[red,double] (k1) arc(180:0:0.75) node[above] (k2) {};
%\draw[red,double] (k1) -- (k1|-x1);
%
%%\node[right,red] at (a) {$\alpha$};
%
%
%
%
%\node[gray] at (w0)[left=5pt] {$p_i$};
%\node[gray] at (w1)[right=5pt] {$p_j$};
%\foreach \n in {w0,w1}
%  \node at (\n)[gray,circle,fill,inner sep=3pt]{};
%
%
%
%\end{tikzpicture} }}\right)

\end{itemize}
\end{example}

\subsection{A concrete case and relation with Section \ref{cBKL}}\label{Concretecase3strands}

We recall that there exists a family of multiforks generating a submodule of $\Habs_m$ for all $m$, and that in the case $m=2$, Bigelow's standard forks provide another basis (of the {\em multiforks submodule}). The definition of multiforks is given in \cite[Section~7.1]{Jules1}. The module generated by (multi)-forks is a strict submodule of $\Habs_m$, see Corollary 7.2 from \cite{Jules1}; this is a consequence of the fact that there exists a diagonal matrix sending the family of code sequences to that of multiforks, but with (non invertible) quantum factorials on the diagonal terms, see \cite[Corollary~7.2]{Jules1} for the precise coefficients. Nevertheless, matrices from Proposition \ref{actioncBKL} can be recovered by those from Proposition \ref{cLawaction}. In this section we study the example with $3$ punctures, we compute matrices from both set-ups and we discuss how they are related. This should help the reader understanding notations and dealing with higher cases.

From now on, $m=2$ so that we study the representation over $\Habs_2$. The braid group $\CB_3$ has two generators $\sigma_1$  and $\sigma_2$. Let $L_1$ and $L_2$ be their representations from Proposition \ref{Lawaction}, the uncolored version for Lawrence representations, matrices names are introduced in Remark \ref{coloringLaw}. Then:
\[
L_1(s,t) = \begin{pmatrix} s^2 t^{-1} & -s^2 (1+ t^{-1}) & s^2 \\ 0 & -s & s \\ 0 & 0 & 1 \end{pmatrix} \text{ and } L_2(s,t) = \begin{pmatrix} 1 & 0 & 0 \\ 1 & -s & 0 \\ 1 & -s(1+t^{-1}) & s^2 t^{-1} \end{pmatrix} .
\]
Let $BKL_1$ and $BKL_2$ the generators' representations from Proposition \ref{actioncBKL}, matrices notations from Remark \ref{relationcBKL}. Then:
\[
BKL_1(q,t)= \begin{pmatrix} q^2 t & 0 & q^2-q \\ 0 & 0 & q \\ 0 & 1 & 1-q  \end{pmatrix} \text{ and } BKL_2(q,t) = \begin{pmatrix} 0 & q & 0 \\ 1 & 1-q & 0 \\ 0 & t(q^2-q) & q^2 t  \end{pmatrix} .
\] 
First, one can check that the two following relations hold, they correspond to Remarks \ref{relationcBKL} and \ref{coloringLaw} adapted to this case.
\[
BKL_1(q_2,t) BKL_2(q_1,t) BKL_1(q_1,t) = BKL_2(q_1,t) BKL_1(q_1,t) BKL_2(q_2,t),
\] 
and
\[
L_1(s_2,t) L_2(s_1,t) L_1(s_1,t) = L_2(s_1,t) L_1(s_1,t) L_2(s_2,t).
\]
This relations should help one with computation of colored version matrices for pure braid groups elements. Then let:
\[
P:= \begin{pmatrix} 1+t & 1+t & 0 \\ 0 & 1+t & 0 \\ 0 & 1+t & 1+t \end{pmatrix},
\]
such that one can check:
\[
P^{-1} L_1(s,t) P = BKL_1 (s,t^{-1}) \text{ while } P^{-1} L_2(s,t) P = BKL_1 (s,t^{-1}).
\]
This emphasizes how to recover matrices from Proposition \ref{actioncBKL} out of those from Proposition \ref{Lawaction}, namely by the change of bases given by $P$ and with the relations between variables $q=s$ and $t=t^{-1}$ (the second one is due to different choices for winding numbers in the literature). One notices that the matrix $P$ is not diagonal as announced in \cite[Corollary~7.2]{Jules1}, for passing from code sequences to forks. This is because forks generators used in Section \ref{cBKL} (taken from \cite{Big1}) do not fit perfectly with multiforks for higher Lawrence representations, see Section~7.1 in \cite{Jules1}, and Remark \ref{forkissue} of the present paper. In the present case, the fork denoted $F_{1,3}$ in Section \ref{cBKL} has the following decomposition:
\[
F_{1,3} = F(2,0) + (1+t) F(1,1) + F(0,2)
\]
where $F(i,j)$ are standard multiforks from \cite[Section~7.1]{Jules1}. This decomposition can be computed from \cite[Example~4.5]{Jules1} and the handle rule, Remark \ref{handleruless}.

\section{Appendix: colored vs pure}\label{pureandcolored}

We have seen in Section \ref{VermaBraiding} as in Definition \ref{coloredBKLrep} that to pass to colored version for representations, one needs to associate one variable per puncture in such a way that if the punctures are permuted variables have to be transported. It seems that the first way to handle this issue is to restrict to a representation of the pure braid group for which punctures are fixed pointwise. Then by means of induced representation one can obtain a representation of the entire braid group for which generators are simpler. This section is devoted to define this induced representation and then to present an object appropriate to this set-up, the {\em colored braid groupoid}. 

\subsection{Induced representation}

Passing from a representation of $\PBn$ to one of $\Bn$ uses the concept of induced representation that we present in this section. 

\begin{defn}[Representation of the braid group]
A {\em representation of $\Bn$} is an algebra morphism:
\[
A \left[ \Bn \right] \to \End_{A}(V)
\]
where $A$ is a ring and $V$ is an $A$-module. 
\end{defn}

\begin{defn}[Induced representation from the pure braid group]\label{inducedpure}
Let $r$ be a representation of $\PBn$:
\[
r: A \left[ \PBn \right] \to \End_{A}(V)
\]
There exists a natural induced representation $Ind(r)$ of $\Bn$ over the space:
\[
Ind(V) = A \left[ \Bn \right] \otimes_{A \left[ \PBn \right]} V
\]
where the action of $\PBn$ is given by product on the left of the tensor product and by $r$ on the right. Its dimension is $n! \times \dim(V)$.

With this notation, the action of $\PBn$ on $Ind(V)$ stabilizes $1 \otimes V$, which recovers initial representation $r$. 
\end{defn}

\begin{example}
Let $\perm: \BC \lbrack \Bn \rbrack \to \BC \lbrack \Sk_n \rbrack$ be the representation that assigns to a braid the permutation it involves on punctures. It is the induced representation from the trivial representation of $\PBn$.
\end{example}

\begin{rmk}\label{CHV}
In the present paper we deal with three families of modules parametrized by the symmetric group.
\begin{itemize}
\item In Section \ref{VermaBraiding} we introduce a family $V^{\tau}$ (for $\tau \in \Sk_n$),
\item By analogy one can define $\CH^{\tau}$ to be $\CH \otimes \tau$ from Definition \ref{definitioncoloredBKL},
\item In Section \ref{lapartiereine} we introduced modules ${\Habs_m}^{\tau}$ for $\tau \in \Sk_n$ and $n \in \BN$. 
\end{itemize}
\end{rmk}

Let $X$ designates the letter $\CH$ or the letter $V$. There exists a reprensentation of the pure braid group on $X^{()}$ such that:
\[
Ind(X) = \bigoplus_{\tau \in \Sk_n} X^{\tau}.
\]

The induced representation is more convenient for computation as generators of $\Bn$ are simpler than those of $\PBn$, although vectors of the modules involves more complicated notations and formulas and the dimension is bigger. We state a remark about the faithfulness of the induced representation.

\begin{rmk}
For a matrix associated with a braid to be the identity on $Ind(X)$, it must stabilize $X^{()}$, hence be a pure braid. 
\end{rmk}

Subsection \ref{matricesforGassner} is probably the easiest to read as it concerns the Gassner representation, and the links between representations of the pure braid groups and induced representations of the braid groups. Next sections apply same protocols in construction of matrices. 

\subsection{Colored braid groupoid}

We present a point of view allowing to see the colored representations not as an induced representation but as a representation of a single object. Namely it involves the generalization of the notion of representation of groups to that of {\em representation of groupoids}. 

\begin{defn}[Groupoid]
A {\em groupoid} $G$ is a category inside which every morphism is invertible.
\end{defn}

\begin{defn}[Representation of a groupoid]
A {\em representation} of a groupoid $G$ is a functor from $G$ to the category $\Vectcat$ of vector spaces.
\end{defn}

\begin{example}\label{fundamgroupoid}
This notion of groupoid is used in topology to generalize the one of fundamental group.
\begin{itemize}
\item[(i)] The fundamental groupoid $\Pi_1(M)$ of a topological space $M$ is the groupoid whose set of objects is $M$ and whose morphisms from $x$ to $y$ are the homotopy-classes $\left[ \gamma \right]$ of continuous maps $\gamma:[0,1] \to M$ with endpoints map to $x$ and $y$ (which the homotopies are required to fix). Composition is by concatenation (and reparametrization) of representative maps.
\item[(ii)] Let $O$ be a subset of a topological space $M$, there exists a sub-groupoid of the fundamental groupoid:
\[
G=\bigcup_{\alpha \in O} G_{\alpha} \subset \Pi_1(M),
\]
it consists in the groupoid of paths having endpoints in $O$.
\item[(iii)] When $O= \lbrace x \rbrace$ is a single point, then the corresponding sub-groupoid is the fundamental group based in $x$. 
\end{itemize}
\end{example}

All the background regarding links between fundamental groupoid and topology can be found in \cite{Brown}, where one can find the correspondence between topological coverings and the fundamental groupoid. %First, we recall the categorical definition of a groupoid. 

%The colored braid groupoid is a subgroupoid of the fundamental groupoid of the configuration space.

\begin{defn}[Colored braid groupoid]\label{coloredbraidgroupoid}
Let $N \in \BN^*$. The colored braid groupoid on $N$ strands is the groupoid whose set of objects is $\Sk_N$ and morphisms between $\tau_1$ and $\tau_2 \in \Sk_N$ are braids $\beta$ satisfying:
\[
\tau_1  \perm(\beta) = \tau_2 
\]
where $\perm$ is the morphism that sends a braid to its induced permutation.
%Let $ \bn \in \lbrace (n_1,\ldots, n_k) \text{ s.t. } n_1+\cdots +n_k=N\rbrace$, $\CC_{\bn}(X)$ be the configuration space defined in the previous section, $x$ be a chosen base point and $O_x$ its orbit under the $\Sn$-action. The {\em colored braid groupoid} $G=\CB_{\bn}(X,x)$ is the groupoid whose set of objects is $O_x$ and morphisms consist in homotopy classes of paths having endpoints in $O_x$.
%\Sk_{n_1} \times \cdots \times \Sk_{n_k}
%
%One can think of $G$ as a space graded by $O_x$ (possible endpoints) as follows:
% \[
% G=\bigcup_{\alpha,\beta \in O_x} G_{\alpha,\beta}
% \]
% so that composition of morphisms is seen as a multiplication law respecting the grading:
% \[
% G_{\alpha,\beta} \times G_{\beta,\gamma} \to G_{\alpha, \gamma}
% \]
%and given by composition of paths. Up to isomorphism, the groupoid $G$ does not depend on the choice of base point.
\end{defn}

%Let $\CB_N$ be the braid group on $N$ strands, we recall that $\perm$ is the morphism that sends a braid to its involved permutation.
%\[
%\perm : \left\lbrace \begin{array}{rcl}
%\CB_N & \to & \Sk_N \\
%b & \mapsto & \perm(b)
%\end{array} \right. ,
%\]

\begin{rmk}
The braid group is the fundamental group of a configuration space, for $x$ an arbitrary base point:
\[
\CB_N = \pi_1 (C_{n,0},x) 
\]
using Definition \ref{configspaceofthepunctureddisk0} for configuration spaces. See \cite{Bir}.
\end{rmk}

%\begin{rmk}
%Let $G, \bn$ and $O_x$ be as in Definition \ref{coloredbraidgroupoid}.
%\begin{itemize}
%\item When $k=1$, $G$ is nothing more than the braid group on $N$ strands $\CB _N$.
%\item Otherwise, every stratum of $G$ can be seen as a subset of $\CB_N$ consisting of braids corresponding to an appropriate permutation. This is the following one to one association:
%\[
%\Phi_{\alpha,\beta}: \lbrace b \in \CB_N \text{ s.t. } \alpha= \left[ \perm(b) \right] \beta \rbrace \to G_{\alpha, \beta}
%\]
%for $\alpha, \beta \in O_x$, $\left[ \perm(b) \right]$ the image of $\perm(b)$ in the quotient $\Sk_N \left/ \Sk_{n_1} \times \cdots \times \Sk_{n_k} \right.$, and so that the (label respecting) product on $G$ is given by the composition of braids.
%\end{itemize}
%\end{rmk}

%This relation allows one to get generators of the colored braid groupoid from the standard braid generators.

\begin{rmk}\label{generatingmorphisms}
Let $\sigma_i,i=1\ldots, N$ be the standard generators of $\CB_N$. Then the system:
\[
\sigma_i^{\alpha} : \alpha \to \perm(\sigma_i) \alpha
\] 
of morphisms for $\alpha \in \Sk_N$, provides generating morphisms of $G$.
\end{rmk}

%The colored braid groupoid is graded by the possible endpoints of paths. %In the following subsection, we give an interpretation of a representation of a groupoid in the formalism of family of representation of groups.

Using the letter $X$ to designate either the letter $\CH$ or $V$ from Remark \ref{CHV}, we define a representation of the colored braid groupoid as follows. Let $\sigma_i^{\alpha}$ be a generating morphism of the colored braid category ($\alpha \in \Sk_n$), we define its representation:
\[
\bfct
X^{\alpha} & \to & X^{(i,i+1)\alpha} \\
v & \mapsto & \sigma_i \cdot v .
\efct
\]

This is a way to deal with colored representations by consideration of the colored braid groupoid.

\end{document}

%% file: FN3bon.pdf_tex
%% Creator: Inkscape inkscape 0.92.4, www.inkscape.org
%% PDF/EPS/PS + LaTeX output extension by Johan Engelen, 2010
%% Accompanies image file 'FN3bon.pdf' (pdf, eps, ps)
%%
%% To include the image in your LaTeX document, write
%%   \input{<filename>.pdf_tex}
%%  instead of
%%   \includegraphics{<filename>.pdf}
%% To scale the image, write
%%   \def\svgwidth{<desired width>}
%%   \input{<filename>.pdf_tex}
%%  instead of
%%   \includegraphics[width=<desired width>]{<filename>.pdf}
%%
%% Images with a different path to the parent latex file can
%% be accessed with the `import' package (which may need to be
%% installed) using
%%   \usepackage{import}
%% in the preamble, and then including the image with
%%   \import{<path to file>}{<filename>.pdf_tex}
%% Alternatively, one can specify
%%   \graphicspath{{<path to file>/}}
%% 
%% For more information, please see info/svg-inkscape on CTAN:
%%   http://tug.ctan.org/tex-archive/info/svg-inkscape
%%
\begingroup%
  \makeatletter%
  \providecommand\color[2][]{%
    \errmessage{(Inkscape) Color is used for the text in Inkscape, but the package 'color.sty' is not loaded}%
    \renewcommand\color[2][]{}%
  }%
  \providecommand\transparent[1]{%
    \errmessage{(Inkscape) Transparency is used (non-zero) for the text in Inkscape, but the package 'transparent.sty' is not loaded}%
    \renewcommand\transparent[1]{}%
  }%
  \providecommand\rotatebox[2]{#2}%
  \newcommand*\fsize{\dimexpr\f@size pt\relax}%
  \newcommand*\lineheight[1]{\fontsize{\fsize}{#1\fsize}\selectfont}%
  \ifx\svgwidth\undefined%
    \setlength{\unitlength}{513.66443904bp}%
    \ifx\svgscale\undefined%
      \relax%
    \else%
      \setlength{\unitlength}{\unitlength * \real{\svgscale}}%
    \fi%
  \else%
    \setlength{\unitlength}{\svgwidth}%
  \fi%
  \global\let\svgwidth\undefined%
  \global\let\svgscale\undefined%
  \makeatother%
  \begin{picture}(1,0.7857112)%
    \lineheight{1}%
    \setlength\tabcolsep{0pt}%
    \put(0,0){\includegraphics[width=\unitlength,page=1]{FN3bon.pdf}}%
    \put(0.2756921,0.08052208){\color[rgb]{0,0,0}\makebox(0,0)[lt]{\begin{minipage}{0.33195552\unitlength}\raggedright {\color{red} $d_1$}\end{minipage}}}%
    \put(0.45346924,0.08052208){\color[rgb]{0,0,0}\makebox(0,0)[lt]{\begin{minipage}{0.33195552\unitlength}\raggedright {\color{red} $d_2$}\end{minipage}}}%
    \put(0.45976225,0.72555419){\color[rgb]{0,0,0}\makebox(0,0)[lt]{\begin{minipage}{0.63873909\unitlength}\raggedright {\color{green} $N_3$}\end{minipage}}}%
    \put(0.22692138,0.52103181){\color[rgb]{0,0,0}\makebox(0,0)[lt]{\begin{minipage}{0.47669445\unitlength}\raggedright $F'$\end{minipage}}}%
    \put(0.1954564,0.63587899){\color[rgb]{0,0,0}\makebox(0,0)[lt]{\begin{minipage}{0.47669445\unitlength}\raggedright $F$\end{minipage}}}%
    \put(0.36694054,0.54148404){\color[rgb]{0,0,0}\makebox(0,0)[lt]{\begin{minipage}{0.36342052\unitlength}\raggedright $z_1$\end{minipage}}}%
    \put(0,0){\includegraphics[width=\unitlength,page=2]{FN3bon.pdf}}%
    \put(0.2977176,0.44079611){\color[rgb]{0,0,0}\makebox(0,0)[lt]{\begin{minipage}{0.45466894\unitlength}\raggedright $z_1'$\end{minipage}}}%
    \put(0.49752021,0.22054127){\color[rgb]{0,0,0}\makebox(0,0)[lt]{\begin{minipage}{0.45466894\unitlength}\raggedright $z_2'$\end{minipage}}}%
    \put(0.47706798,0.34640118){\color[rgb]{0,0,0}\makebox(0,0)[lt]{\begin{minipage}{0.45466894\unitlength}\raggedright $z_2$\end{minipage}}}%
    \put(0.041278,0.42034388){\color[rgb]{0,0,0}\makebox(0,0)[lt]{\begin{minipage}{0.42320399\unitlength}\raggedright {\color{gray} $p_2$}\end{minipage}}}%
    \put(0.20017615,0.41719737){\color[rgb]{0,0,0}\makebox(0,0)[lt]{\begin{minipage}{0.42320399\unitlength}\raggedright {\color{gray} $p_1$}\end{minipage}}}%
    \put(0.62023362,0.49428658){\color[rgb]{0,0,0}\makebox(0,0)[lt]{\begin{minipage}{0.42320399\unitlength}\raggedright {\color{gray} $p_4$}\end{minipage}}}%
  \end{picture}%
\endgroup%

%% file: deltas.pdf_tex
%% Creator: Inkscape inkscape 0.92.4, www.inkscape.org
%% PDF/EPS/PS + LaTeX output extension by Johan Engelen, 2010
%% Accompanies image file 'deltas.pdf' (pdf, eps, ps)
%%
%% To include the image in your LaTeX document, write
%%   \input{<filename>.pdf_tex}
%%  instead of
%%   \includegraphics{<filename>.pdf}
%% To scale the image, write
%%   \def\svgwidth{<desired width>}
%%   \input{<filename>.pdf_tex}
%%  instead of
%%   \includegraphics[width=<desired width>]{<filename>.pdf}
%%
%% Images with a different path to the parent latex file can
%% be accessed with the `import' package (which may need to be
%% installed) using
%%   \usepackage{import}
%% in the preamble, and then including the image with
%%   \import{<path to file>}{<filename>.pdf_tex}
%% Alternatively, one can specify
%%   \graphicspath{{<path to file>/}}
%% 
%% For more information, please see info/svg-inkscape on CTAN:
%%   http://tug.ctan.org/tex-archive/info/svg-inkscape
%%
\begingroup%
  \makeatletter%
  \providecommand\color[2][]{%
    \errmessage{(Inkscape) Color is used for the text in Inkscape, but the package 'color.sty' is not loaded}%
    \renewcommand\color[2][]{}%
  }%
  \providecommand\transparent[1]{%
    \errmessage{(Inkscape) Transparency is used (non-zero) for the text in Inkscape, but the package 'transparent.sty' is not loaded}%
    \renewcommand\transparent[1]{}%
  }%
  \providecommand\rotatebox[2]{#2}%
  \newcommand*\fsize{\dimexpr\f@size pt\relax}%
  \newcommand*\lineheight[1]{\fontsize{\fsize}{#1\fsize}\selectfont}%
  \ifx\svgwidth\undefined%
    \setlength{\unitlength}{1185.43058328bp}%
    \ifx\svgscale\undefined%
      \relax%
    \else%
      \setlength{\unitlength}{\unitlength * \real{\svgscale}}%
    \fi%
  \else%
    \setlength{\unitlength}{\svgwidth}%
  \fi%
  \global\let\svgwidth\undefined%
  \global\let\svgscale\undefined%
  \makeatother%
  \begin{picture}(1,0.55515074)%
    \lineheight{1}%
    \setlength\tabcolsep{0pt}%
    \put(0,0){\includegraphics[width=\unitlength,page=1]{deltas.pdf}}%
    \put(0.08771708,0.51117141){\color[rgb]{0.50196078,0.50196078,0.50196078}\makebox(0,0)[lt]{\lineheight{0}\smash{\begin{tabular}[t]{l}$p_1$\end{tabular}}}}%
    \put(0,0){\includegraphics[width=\unitlength,page=2]{deltas.pdf}}%
    \put(0.15054715,0.49060636){\color[rgb]{1,0,0}\makebox(0,0)[lt]{\lineheight{0}\smash{\begin{tabular}[t]{l}$d_1$\end{tabular}}}}%
    \put(0.23407987,0.48960628){\color[rgb]{1,0,0}\makebox(0,0)[lt]{\lineheight{0}\smash{\begin{tabular}[t]{l}$d_2$\end{tabular}}}}%
    \put(0,0){\includegraphics[width=\unitlength,page=3]{deltas.pdf}}%
    \put(0.08688664,0.34757725){\color[rgb]{0.50196078,0.50196078,0.50196078}\makebox(0,0)[lt]{\lineheight{0}\smash{\begin{tabular}[t]{l}$p_1$\end{tabular}}}}%
    \put(0,0){\includegraphics[width=\unitlength,page=4]{deltas.pdf}}%
    \put(0.14669243,0.30206129){\color[rgb]{1,0,0}\makebox(0,0)[lt]{\lineheight{0}\smash{\begin{tabular}[t]{l}$d_1$\end{tabular}}}}%
    \put(0.24830302,0.30417233){\color[rgb]{1,0,0}\makebox(0,0)[lt]{\lineheight{0}\smash{\begin{tabular}[t]{l}$d_2$\end{tabular}}}}%
    \put(0,0){\includegraphics[width=\unitlength,page=5]{deltas.pdf}}%
    \put(0.00840349,0.51748743){\color[rgb]{0.50196078,0.50196078,0.50196078}\makebox(0,0)[lt]{\lineheight{0}\smash{\begin{tabular}[t]{l}$p_2$\end{tabular}}}}%
    \put(0,0){\includegraphics[width=\unitlength,page=6]{deltas.pdf}}%
    \put(0.00840349,0.34683467){\color[rgb]{0.50196078,0.50196078,0.50196078}\makebox(0,0)[lt]{\lineheight{0}\smash{\begin{tabular}[t]{l}$p_2$\end{tabular}}}}%
    \put(0,0){\includegraphics[width=\unitlength,page=7]{deltas.pdf}}%
    \put(0.29184648,0.51031504){\color[rgb]{0.50196078,0.50196078,0.50196078}\makebox(0,0)[lt]{\lineheight{0}\smash{\begin{tabular}[t]{l}$p_3$\end{tabular}}}}%
    \put(0,0){\includegraphics[width=\unitlength,page=8]{deltas.pdf}}%
    \put(0.29101602,0.34672086){\color[rgb]{0.50196078,0.50196078,0.50196078}\makebox(0,0)[lt]{\lineheight{0}\smash{\begin{tabular}[t]{l}$p_3$\end{tabular}}}}%
    \put(0,0){\includegraphics[width=\unitlength,page=9]{deltas.pdf}}%
    \put(0.38046901,0.51099676){\color[rgb]{0.50196078,0.50196078,0.50196078}\makebox(0,0)[lt]{\lineheight{0}\smash{\begin{tabular}[t]{l}$p_4$\end{tabular}}}}%
    \put(0,0){\includegraphics[width=\unitlength,page=10]{deltas.pdf}}%
    \put(0.37963857,0.34740256){\color[rgb]{0.50196078,0.50196078,0.50196078}\makebox(0,0)[lt]{\lineheight{0}\smash{\begin{tabular}[t]{l}$p_4$\end{tabular}}}}%
    \put(0,0){\includegraphics[width=\unitlength,page=11]{deltas.pdf}}%
    \put(0.6589381,0.50755608){\color[rgb]{0.50196078,0.50196078,0.50196078}\makebox(0,0)[lt]{\lineheight{0}\smash{\begin{tabular}[t]{l}$p_1$\end{tabular}}}}%
    \put(0,0){\includegraphics[width=\unitlength,page=12]{deltas.pdf}}%
    \put(0.72176819,0.48699103){\color[rgb]{1,0,0}\makebox(0,0)[lt]{\lineheight{0}\smash{\begin{tabular}[t]{l}$d_1$\end{tabular}}}}%
    \put(0.80530091,0.48599095){\color[rgb]{1,0,0}\makebox(0,0)[lt]{\lineheight{0}\smash{\begin{tabular}[t]{l}$d_2$\end{tabular}}}}%
    \put(0,0){\includegraphics[width=\unitlength,page=13]{deltas.pdf}}%
    \put(0.65810767,0.34396192){\color[rgb]{0.50196078,0.50196078,0.50196078}\makebox(0,0)[lt]{\lineheight{0}\smash{\begin{tabular}[t]{l}$p_1$\end{tabular}}}}%
    \put(0,0){\includegraphics[width=\unitlength,page=14]{deltas.pdf}}%
    \put(0.71791348,0.29844595){\color[rgb]{1,0,0}\makebox(0,0)[lt]{\lineheight{0}\smash{\begin{tabular}[t]{l}$d_1$\end{tabular}}}}%
    \put(0.81952408,0.300557){\color[rgb]{1,0,0}\makebox(0,0)[lt]{\lineheight{0}\smash{\begin{tabular}[t]{l}$d_2$\end{tabular}}}}%
    \put(0,0){\includegraphics[width=\unitlength,page=15]{deltas.pdf}}%
    \put(0.57962452,0.51387211){\color[rgb]{0.50196078,0.50196078,0.50196078}\makebox(0,0)[lt]{\lineheight{0}\smash{\begin{tabular}[t]{l}$p_2$\end{tabular}}}}%
    \put(0,0){\includegraphics[width=\unitlength,page=16]{deltas.pdf}}%
    \put(0.57962452,0.34321934){\color[rgb]{0.50196078,0.50196078,0.50196078}\makebox(0,0)[lt]{\lineheight{0}\smash{\begin{tabular}[t]{l}$p_2$\end{tabular}}}}%
    \put(0,0){\includegraphics[width=\unitlength,page=17]{deltas.pdf}}%
    \put(0.86306751,0.50669972){\color[rgb]{0.50196078,0.50196078,0.50196078}\makebox(0,0)[lt]{\lineheight{0}\smash{\begin{tabular}[t]{l}$p_3$\end{tabular}}}}%
    \put(0,0){\includegraphics[width=\unitlength,page=18]{deltas.pdf}}%
    \put(0.86223705,0.34310554){\color[rgb]{0.50196078,0.50196078,0.50196078}\makebox(0,0)[lt]{\lineheight{0}\smash{\begin{tabular}[t]{l}$p_3$\end{tabular}}}}%
    \put(0,0){\includegraphics[width=\unitlength,page=19]{deltas.pdf}}%
    \put(0.95169001,0.50738144){\color[rgb]{0.50196078,0.50196078,0.50196078}\makebox(0,0)[lt]{\lineheight{0}\smash{\begin{tabular}[t]{l}$p_4$\end{tabular}}}}%
    \put(0,0){\includegraphics[width=\unitlength,page=20]{deltas.pdf}}%
    \put(0.95085959,0.34378725){\color[rgb]{0.50196078,0.50196078,0.50196078}\makebox(0,0)[lt]{\lineheight{0}\smash{\begin{tabular}[t]{l}$p_4$\end{tabular}}}}%
    \put(0,0){\includegraphics[width=\unitlength,page=21]{deltas.pdf}}%
    \put(0.36067397,0.2110996){\color[rgb]{0.50196078,0.50196078,0.50196078}\makebox(0,0)[lt]{\lineheight{0}\smash{\begin{tabular}[t]{l}$p_1$\end{tabular}}}}%
    \put(0,0){\includegraphics[width=\unitlength,page=22]{deltas.pdf}}%
    \put(0.42350405,0.19053455){\color[rgb]{1,0,0}\makebox(0,0)[lt]{\lineheight{0}\smash{\begin{tabular}[t]{l}$d_1$\end{tabular}}}}%
    \put(0.50703677,0.18953447){\color[rgb]{1,0,0}\makebox(0,0)[lt]{\lineheight{0}\smash{\begin{tabular}[t]{l}$d_2$\end{tabular}}}}%
    \put(0,0){\includegraphics[width=\unitlength,page=23]{deltas.pdf}}%
    \put(0.35984354,0.04750544){\color[rgb]{0.50196078,0.50196078,0.50196078}\makebox(0,0)[lt]{\lineheight{0}\smash{\begin{tabular}[t]{l}$p_1$\end{tabular}}}}%
    \put(0,0){\includegraphics[width=\unitlength,page=24]{deltas.pdf}}%
    \put(0.41964932,0.00198949){\color[rgb]{1,0,0}\makebox(0,0)[lt]{\lineheight{0}\smash{\begin{tabular}[t]{l}$d_1$\end{tabular}}}}%
    \put(0.52125991,0.00410053){\color[rgb]{1,0,0}\makebox(0,0)[lt]{\lineheight{0}\smash{\begin{tabular}[t]{l}$d_2$\end{tabular}}}}%
    \put(0,0){\includegraphics[width=\unitlength,page=25]{deltas.pdf}}%
    \put(0.28136039,0.21741563){\color[rgb]{0.50196078,0.50196078,0.50196078}\makebox(0,0)[lt]{\lineheight{0}\smash{\begin{tabular}[t]{l}$p_2$\end{tabular}}}}%
    \put(0,0){\includegraphics[width=\unitlength,page=26]{deltas.pdf}}%
    \put(0.28136039,0.04676287){\color[rgb]{0.50196078,0.50196078,0.50196078}\makebox(0,0)[lt]{\lineheight{0}\smash{\begin{tabular}[t]{l}$p_2$\end{tabular}}}}%
    \put(0,0){\includegraphics[width=\unitlength,page=27]{deltas.pdf}}%
    \put(0.56480337,0.21024324){\color[rgb]{0.50196078,0.50196078,0.50196078}\makebox(0,0)[lt]{\lineheight{0}\smash{\begin{tabular}[t]{l}$p_3$\end{tabular}}}}%
    \put(0,0){\includegraphics[width=\unitlength,page=28]{deltas.pdf}}%
    \put(0.56397292,0.04664906){\color[rgb]{0.50196078,0.50196078,0.50196078}\makebox(0,0)[lt]{\lineheight{0}\smash{\begin{tabular}[t]{l}$p_3$\end{tabular}}}}%
    \put(0,0){\includegraphics[width=\unitlength,page=29]{deltas.pdf}}%
    \put(0.65342591,0.21092496){\color[rgb]{0.50196078,0.50196078,0.50196078}\makebox(0,0)[lt]{\lineheight{0}\smash{\begin{tabular}[t]{l}$p_4$\end{tabular}}}}%
    \put(0,0){\includegraphics[width=\unitlength,page=30]{deltas.pdf}}%
    \put(0.65259547,0.04733076){\color[rgb]{0.50196078,0.50196078,0.50196078}\makebox(0,0)[lt]{\lineheight{0}\smash{\begin{tabular}[t]{l}$p_4$\end{tabular}}}}%
    \put(0,0){\includegraphics[width=\unitlength,page=31]{deltas.pdf}}%
    \put(0.18004393,0.56021245){\color[rgb]{0,0,0}\makebox(0,0)[lt]{\begin{minipage}{0.39226253\unitlength}\raggedright $\delta_{1,1}$\end{minipage}}}%
    \put(0.76391858,0.55659714){\color[rgb]{0,0,0}\makebox(0,0)[lt]{\begin{minipage}{0.58929763\unitlength}\raggedright $\delta_{2,2}$\end{minipage}}}%
    \put(0.47830811,0.26737129){\color[rgb]{0,0,0}\makebox(0,0)[lt]{\begin{minipage}{0.66160404\unitlength}\raggedright $\delta_{2,1}$\end{minipage}}}%
  \end{picture}%
\endgroup%

%% file: N23.pdf_tex
%% Creator: Inkscape inkscape 0.92.4, www.inkscape.org
%% PDF/EPS/PS + LaTeX output extension by Johan Engelen, 2010
%% Accompanies image file 'N23.pdf' (pdf, eps, ps)
%%
%% To include the image in your LaTeX document, write
%%   \input{<filename>.pdf_tex}
%%  instead of
%%   \includegraphics{<filename>.pdf}
%% To scale the image, write
%%   \def\svgwidth{<desired width>}
%%   \input{<filename>.pdf_tex}
%%  instead of
%%   \includegraphics[width=<desired width>]{<filename>.pdf}
%%
%% Images with a different path to the parent latex file can
%% be accessed with the `import' package (which may need to be
%% installed) using
%%   \usepackage{import}
%% in the preamble, and then including the image with
%%   \import{<path to file>}{<filename>.pdf_tex}
%% Alternatively, one can specify
%%   \graphicspath{{<path to file>/}}
%% 
%% For more information, please see info/svg-inkscape on CTAN:
%%   http://tug.ctan.org/tex-archive/info/svg-inkscape
%%
\begingroup%
  \makeatletter%
  \providecommand\color[2][]{%
    \errmessage{(Inkscape) Color is used for the text in Inkscape, but the package 'color.sty' is not loaded}%
    \renewcommand\color[2][]{}%
  }%
  \providecommand\transparent[1]{%
    \errmessage{(Inkscape) Transparency is used (non-zero) for the text in Inkscape, but the package 'transparent.sty' is not loaded}%
    \renewcommand\transparent[1]{}%
  }%
  \providecommand\rotatebox[2]{#2}%
  \newcommand*\fsize{\dimexpr\f@size pt\relax}%
  \newcommand*\lineheight[1]{\fontsize{\fsize}{#1\fsize}\selectfont}%
  \ifx\svgwidth\undefined%
    \setlength{\unitlength}{513.66444552bp}%
    \ifx\svgscale\undefined%
      \relax%
    \else%
      \setlength{\unitlength}{\unitlength * \real{\svgscale}}%
    \fi%
  \else%
    \setlength{\unitlength}{\svgwidth}%
  \fi%
  \global\let\svgwidth\undefined%
  \global\let\svgscale\undefined%
  \makeatother%
  \begin{picture}(1,0.78571115)%
    \lineheight{1}%
    \setlength\tabcolsep{0pt}%
    \put(0,0){\includegraphics[width=\unitlength,page=1]{N23.pdf}}%
    \put(0.27569211,0.08052208){\color[rgb]{0,0,0}\makebox(0,0)[lt]{\begin{minipage}{0.33195551\unitlength}\raggedright {\color{red} $d_1$}\end{minipage}}}%
    \put(0.45346924,0.08052208){\color[rgb]{0,0,0}\makebox(0,0)[lt]{\begin{minipage}{0.33195551\unitlength}\raggedright {\color{red} $d_2$}\end{minipage}}}%
    \put(0.31633461,0.68757259){\color[rgb]{0,0,0}\makebox(0,0)[lt]{\begin{minipage}{0.63873908\unitlength}\raggedright {\color{green} $N_{2,3}$}\end{minipage}}}%
    \put(0.04127802,0.42034387){\color[rgb]{0,0,0}\makebox(0,0)[lt]{\begin{minipage}{0.42320398\unitlength}\raggedright {\color{gray} $p_2$}\end{minipage}}}%
    \put(0.20017616,0.41719737){\color[rgb]{0,0,0}\makebox(0,0)[lt]{\begin{minipage}{0.42320398\unitlength}\raggedright {\color{gray} $p_1$}\end{minipage}}}%
    \put(0.62023362,0.49428657){\color[rgb]{0,0,0}\makebox(0,0)[lt]{\begin{minipage}{0.42320398\unitlength}\raggedright {\color{gray} $p_4$}\end{minipage}}}%
  \end{picture}%
\endgroup%

%% file: article2ter.bbl
\begin{thebibliography}{99999}

\baselineskip15pt

\bibitem[B-N]{B-N}
D.Bar-Natan, {\em  A Note on the Unitarity Property of the Gassner Invariant},  Bulletin of Chelyabinsk State University (Mathematics, Mechanics, Informatics) {\bf 3-358-17} (2015), 22--25.

\bibitem[Big]{Big1}
S. Bigelow, {\em Braid groups are linear}, J. Amer. Math. Soc. {\bf 14} (2000), 471--486.


\bibitem[Big1]{Big0}
S. Bigelow, {\em Homological representations of the Iwahori-Hecke algebra}, Geom. Topol. Monographs {\bf 7}, (2004), 493--507.

\bibitem[Bir]{Bir}
J. Birman, {\em Braids, Links and Mapping Class Groups.}, Annals of Mathematics Studies {\bf 82} Princeton University Press, Princeton, New Jersey, (1975).

\bibitem[Br]{Brown}
R. Brown, {\em Topology and Groupoids}, Booksurge, (2006).

\bibitem[F-M]{F-M}
B. Farb, D. Margalit, {\em A Primer on Mapping Class Groups}, Princeton University Press (2012).

\bibitem[Hab]{Hab}
K. Habiro, {\em An integral form of the quantized enveloping algebra of sl2 and its completions}, J. Pure Appl. Alg. {\bf 211} (2007), 265--292.

\bibitem[Ito]{Itogarside}
T. Ito, {\em Reading the dual Garside length of braids from homological and quantum representations}, Comm. Math. Phys. {\em 335} (2015) 345--367.

\bibitem[Ito2]{Ito2}
T. Ito, {\em Topological formula of the loop expansion of the colored Jones polynomials}, Trans. Amer. Math. Soc., (2019). 

\bibitem[J-K]{JK}
C. Jackson and T. Kerler, {\em The  Lawrence-Krammer-Bigelow  representations  of  the  braid groups via $\Uq$}, Adv. Math, {\bf 228}, (2011), 1689--1717.

\bibitem[K-T]{K-T}
C. Kassel, V. Turaev, {\em Braid Groups}, Springer (2008).

\bibitem[Knu]{Knu}
K. Knudson, {\em On the kernel of the Gassner representation}, Archiv der Mathematik {\bf 85}, (2005), 108--117. 

\bibitem[K1]{Kohmulti}
T. Kohno, {\em Hyperplane arrangements, local system homology and iterated integrals},Arrangements of Hyperplanes—Sapporo 2009, 157--174, Mathematical Society of Japan, Tokyo, Japan, 2012. 

\bibitem[K2]{Koh}
T. Kohno, {\em Quantum and homological representations of braid groups}, Configuration Spaces - Geometry, Combinatorics and Topology, Edizioni della Normale (2012), 355--372.

\bibitem[Kra]{Kra}
D. Krammer, {\em Braid groups are linear}, Ann. Math. {\bf 155} (2002), 131--156.

\bibitem[Law]{Law}
R. Lawrence, {\em Homological representations of the Hecke algebra}, Comm. Math. Phys. {\bf 135} (1990), 141--191.

\bibitem[Lus]{Lus}
G. Lusztig, {\em Finite dimensional Hopf algebras arising from quantum groups} , J. Am. Mat. Soc. {\bf 3} (1990) , 257--296.


\bibitem[M-W]{MW}
I. Marin, E. Wagner, {\em A cubic defining algebra for the Links-Gould polynomial}, Advances in Mathematics , {\bf 248}, (2013), 1332--1365.

\bibitem[M]{Jules1}
J. Martel, {\em A homological model for $\Uq$ Verma-modules and their braid representations}, arXiv:2002.08785, math.GT.

\bibitem[Mo]{Mo}
H.R. Morton, {\em The Multivariable Alexander Polynomial for a Closed Braid}, Contemporary Mathematics {\bf 233}, Amer. Math. Soc. (1999), 167--172.

\bibitem[P-P]{P-P}
L. Paoluzzi, L. Paris, {\em A  note  on  the  Lawrence-Krammer-Bigelow representation}, Algebr. Geom. Topol., {\bf 2} (2002), 499--518.

\bibitem[Zi]{Zi}
M. Zinno, {\em On Krammer’s representation of the braid group}, Math. Ann. {\bf 321}, (2001), 197 -- 211.

\end{thebibliography}
